\documentclass{article}
\usepackage{amsmath}
\usepackage{amssymb}

\newcommand{\C}{\mathbb{C}}
\newcommand{\Z}{\mathbb{Z}}
\newcommand{\N}{\mathbb{N}}
\newcommand{\TT}{\mathbb{T}}

\def\co{\colon}
\def\la{\langle}
\def\ra{\rangle}
\def\i{\infty}
\def\imp{\Longrightarrow}
\def\op{\oplus}

\def\ol{\overline}

\def\HH{{\cal H}}
\def\L{{\cal L}}
\def\K{{\cal K}}
\def\I{{\cal I}}
\def\G{{\cal G}}

\def\E{{\cal E}}
\def\M{{\cal M}}
\def\F{{\cal F}}

\def\X{{\cal X}}
\def\T{{\cal T}}
\def\S{{\cal S}}

\def\J{{\cal J}}

\def\R{{\cal R}}

\def\o{\omega}

\def\z{\zeta}

\def\s{\sigma}
\def\et{\eta}

\def\d{\delta}
\def\D{\Delta}

\def\m{\mu}

\def\x{\xi}

\def\1{\mathbf 1}

\def\wt{\widetilde}

\def\wh{\widehat}
\def\ol{\overline}

\def\pd{\overset{.}+}
\def\ot{\otimes}

\def\DD{{\cal D}}
\def\G{\Gamma}
\def\fa{\forall}

\newtheorem{pro}{Proposition}
\newtheorem{thm}[pro]{Theorem}
\newtheorem{lem}[pro]{Lemma}
\newtheorem{cor}[pro]{Corollary}

\newtheorem{que}[pro]{Question}

\begin{document}

\title{On hyperinvariant subspaces of operators containing unilateral shifts}

\author{L\'aszl\'o K\'erchy}

\date{\emph{Dedicated to the memory of Franciszek Hugon Szafraniec, outstanding representative of the Krakow school of operator theory}}

\maketitle

\begin{abstract}
\noindent
The hyperinvariant subspace problem for Hilbert space operators $T$ containing a unilateral shift is addressed.
The discussion is based on a similarity model of $T$, which is an operator-matrix $\wh T= [T_{i,j}]_3$ where $T_{1,1}$ is the simple unilateral shift $S$ and $T_{3,3}$ is a cyclic diagonal operator $D$. 
The existence of $D$ is established by the technique resulting almost invariant half-spaces in \cite{APTT}; see also \cite{Tc} and \cite{HP}.
For any operator $Q=[Q_{i,j}]_3$ in the commutant of $\wh T$, the entry $Q_{3,1}$ intertwines $S$ with $D$ up to a transformation of rank at most 1.
These entries form a linear manifold $\L_{3,1}$.
We focus on 3-dimensional cross-sections of $\L_{3,1}$.
These are subspaces of $3\times 3$ complex matrices, transformed into singular matrices by a canonical mapping.
If such a subspace $\L$ is not transitive, then $T$ has a nontrivial hyperinvariant subspace.
A throrough study reveals that $\L$ can be transitive only if it has a very specific basis.
Consequences of the existence of nontrivial hyperinvariant subspaces in the presence of shift-type invariant subspaces are also discussed.

\noindent
\emph{AMS Subject Classification} (2020): 47A08, 47A15, 47A20, 47L05.

\noindent
\emph{Key words}: invariant subspace, hyperinvariant subspace, unilateral shift, transitive subspace of matrices, singular matrices.
\end{abstract}

\section{Introduction}
\label{introduction}

Let $\HH$ be a (nonzero) complex Hilbert space, and let $\L(\HH)$ denote the algebra of bounded, linear operators acting on $\HH$.
The \emph{invariant subspace problem} (ISP) asks whether every $T\in \L(\HH)$ has a nontrivial invariant subspace $\M:\; \{0\}\ne\M\ne\HH,\; T\M\subset\M$.
Taking any $0\ne x\in\HH$ and considering $\M_x=\vee\{T^nx\}_{n=0}^\infty$, this question can be reduced to the separable case.
On the other hand, in finite dimension the operators have eigenvalues and eigenspaces.
Hence we may (and shall) assume that $\dim \HH= \aleph_0$.
The \emph{hyperinvariant subspace problem} (HSP) asks whether every nonscalar $T\in\L(\HH)$ has a nontrivial hyperinvariant subspace, that is a subspace which is invariant for every operator $Q$ commuting with $T:\; QT=TQ$.
The commutant of $T$ is denoted by $\{T\}'$.
These questions are arguably the most challenging open problems in operator theory.
For a thorough study of this topic we refer to the paper
\cite{DP} and the monographs \cite{RR} and \cite{CP}.

It is known that the answer for (ISP) is negative in the general complex Banach space setting; see \cite{Enflo} and \cite{Read}.
On the other hand, in particular Banach spaces the answer for (HSP) is positive; see \cite{AH}.
A recent approach to (ISP) in an arbitrary (infinite dimensional) Banach space $\X$ was initiated in \cite{APTT}, and culminated in \cite{Tc} proving that every operator $T$ on $\X$ has an \emph{almost invariant half-space}. 
We recall that $\M$ is a half-space in $\X$ if $\M$ and $\X/\M$ are infinite dimensional.
Furthermore, $\M$ is almost invariant for $T$ if $\M$ is invariant for a finite rank perturbation of $T$.
The technique of this approach was refined in the Hilbert space setting in \cite{JKP}, which led to the structure theorem in the paper \cite{HP}.
Applying this result to block-triangular operators  more delicate structure theorem was constructed in the recent paper \cite{ker25}.
Relying on it new sufficient conditions were obtained for the existence of proper hyperinvariant subspaces by studying cross-sections of the commutant.

In the present paper we concentrate on operators which have shift-type invariant subspaces.
We recall that $\HH_1$ is a \emph{shift-type invariant subspace} of the operator $T\in\L(\HH)$, if the restriction $T|_{\HH_1}$ is similar to the simple unilateral shift $S$.
The system of shift-type invariant subspaces is denoted by $\hbox{Lat}_sT$.

We address the following

\begin{que}
\label{main-question}
Is it true that if a Hilbert space operator has a shift-type invariant subspace, then it has a nontrivial hyperinvariant subspace?
\end{que}

An affirmative answer would yield positive answers for longstanding open questions on the existence of proper hyperinvariant and invariant subspaces, as it is presented in Section 2.

Our study of this problem is based on a similarity model $\wh T$ of a translate of $T$, constructed in Section 3.
Here $\wh T=[T_{i,j}]_3$ is an operator-matrix where $T_{1,1}$ is the unilateral shift $S$ and $T_{3,3}$ is a cyclic diagonal operator $D$ derived by the technique of almost invariant half-spaces.
For every operator $Q=[Q_{i,j}]_3$ in the commutant of $\wh T$, the entry $Q_{3,1}$ intertwines $S$ with $D$ up to a transformation of rank at most 1.
These entries form the linear manifold $\L_{3,1}$.
Assuming that every nonzero vector $x\in\HH_1$ is cyclic for the commutant $\{T\}'$, we may infer that $\L_{3,1}$ is transitive.

In Section 4 we turn to the study of 3-dimensional cross-sections $\L_*$ of $\L_{3,1}$, which inherit transitivity and the singularity property of $\L_{3,1}$.
Thus, if $\L_*$ is not transitive then $T$ has a nontrivial hyperinvariant subspace.
In Section 5 we proceed with a systematic investigation of subspaces $\L$ of $3 \times 3$ complex matrices, which are transformed into singular matrices by a mapping induced by a cyclic $3 \times 3$ diagonal matrix $\D$.
The dimension of a $\D$-singular subspace $\L$ which is also transitive can be 5 and 6.
In Section 6 it is shown that for 6-dimensional subspaces $\D$-singularity and transitivity are inconsistent properties.
The 5-dimensional situation is much more complicated.
In Section 7 we focus on the subspace $\L_0$ consisting of those matrices in $\L$ where the third column is zero.
The cases when $\L_0$ possesses a property, which is consistent with the $\D$-singularity of $\L$,
 are listed.
In Section 8 it is shown that if $\L$, containing such an $\L_0$, is really $\D$-singular then it cannot be transitive, with a special exceptional situation.
The $\D$-singularity of $\L$ in this exceptional case is characterized in Section 9 in terms of a canonical basis.
Then, in Section 10 it is shown that transitivity also holds when $\D$-singularity is valid in this particular situation.
Finally, in Section 11 we conclude that
 the hyperinvariant subspace problem for operators containing a unilateral shift can be reduced to the study of the case when the previous 3-dimensional cross-sections have a very special canonical basis.

As for notation, we shall frequently use the Kronecker symbol $\d_{k,l}$, which is 1 if $k=l$, and is 0 if $k\ne l$.

\section{Nontrivial hyperinvariant subspaces}
\label{hyperinvariant}

In this section we present consequences of an affirmative answer for Question \ref{main-question}.
First we consider \emph{Hilbert space contractions}; for their Sz.-Nagy--Foias theory we refer to \cite{NFBK}.

So let us assume that $T\in\L(\HH)$ and $\|T\|\le1$.
By its canonical decomposition $T$ splits into the orthogonal sum $T= U_s \op U_a \op T_c$, where $U_s$ is a singular unitary operator, $U_a$ is an absolutely continuous unitary operator and $T_c$ is a completely nonunitary contraction.
The hyperinvariant subspace lattice $\hbox{Hlat}\,T$ of $T$ also splits into the direct sum
\[\hbox{Hlat}\,T= \hbox{Hlat}\,U_s \op \hbox{Hlat}(U_a\op T_c).\]
The lattice $\hbox{Hlat }U_s$ consists of the spectral subspaces of $U_s$.
Hence it can be assumed that $U_s$ acts on the zero space, when the contraction $T$ is called \emph{absolutely continuous} (a.c.).

It is easy to see that the subspace 
\[\HH_0(T)= \left\{ h\in\HH: \lim_{n\to\i}\|T^nh\|=0\right\}\]
of stable vectors is hyperinvariant for $T$.
We say that $T$ is \emph{asymptotically nonvanishing}, if $\HH_0(T)\ne \HH$.
In view of (HSP) we may assume then that $\HH_0(T)=\{0\}$, in notation: $T\in C_{1\cdot}$.

We recall also the concept of unitary asymptote.
Given any operator $R\in\L(\K)$, the intertwining set $\I(T,R)$ stands for the system of transformations $X\in\L(\HH,\K)$ satisfying the condition $XT=RX$.
We say that $(X,U)$ is a contractive unitary intertwining pair for $T$, if $U\in\L(\K)$ is a unitary operator and $X\in\I(T,U)$ with $\|X\|\le 1$.
The contractive unitary intertwining pair $(X_T,U_T)$ is called a \emph{unitary asymptote} of $T$, if it is universal in the sense that for every contractive unitary intertwining pair $(X,U)$ of $T$ there exists a unique $Y\in \I(U_T,U)$ such that $X=YX_T$ and $\|Y\|\le 1$.
Such a pair exists and is unique up to isomorphism.
It can be characterized by the properties
\[X_T\in\I(T,U_T), \;\; \|X_Th\|=\lim_{n\to\i}\|T^nh\|\;(h\in\HH), \;\; \K_T= \vee\{U_T^{-n}X_T\HH\}_{n=1}^\i.\]
(See \cite{ker13}.)

If $T$ is an a.c.\ contraction, then $U_T$ is an a.c.\ unitary operator.
The \emph{residual set} $\omega(T)$ is the measurable support of the spectral measure $E_T$ of $U_T$.
This means that for any Borel subset $\o$ of the unit circle $\TT$, $E_T(\o)\ne0$ holds exactly when $m_\circ(\o\cap\o(T))\ne0$.
(Here $m_\circ$ stands for the normalized Lebesgue measure on $\TT$.)
It is known that if $\o(T)=\TT$, then $T$ has shift-type invariant subspaces; actually, the subspaces in $\hbox{Lat}_sT$ span the whole space $\HH$.
(See \cite{ker07} and Chapter IX in \cite{NFBK}.)

Immediate consequence of a positive answer for Question \ref{main-question} is

\begin{cor}
\label{full-residual}
If $T$ is an asymptotically nonvanishing a.c.\ contraction with a residual set $\o(T)=\TT$, then $T$ has a nontrivial hyperinvariant subspace.
\end{cor}

The asymptotically nonvanishing a.c.\  contraction $T$ is called \emph{quasianalytic}, if given any nonzero $h\in\HH$, the vector $X_Th$ is cyclic for $\{U_T\}'$.
It turns out that (HSP) can be reduced to this special class.
We note also that quasianalytic contractions are of class $C_{10}$:
\[\lim_{n\to\i}\|T^{*n}h\|=0< \lim_{n\to\i}\|T^nh\|,\quad \hbox{ for every } \; 0\ne h\in\HH.\]
(See \cite{ker13}.)

The contraction $T$ is called \emph{asymptotically cyclic}, if $U_T$ is cyclic.
In that case the commutant $\{U_T\}'$ can be identified with a function algebra; see \cite{ker11} and \cite{KSz}.
Let $\L_0(\HH)$ denote the set of asymptotically cyclic, quasianalytic contractions, and let $\L_1(\HH)$ be the set of contractions $T\in\L_0(\HH)$ with $\o(T)=\TT$.
It turns out that (HSP) is equivalent in these classes.
Namely, it has been proved in \cite{KT} that for every $T\in\L_0(\HH)$ there exists an $R\in\L_1(\HH)$ such that $\{T\}'=\{R\}'$, and so $\hbox{Hlat}\,T= \hbox{Hlat}\,R$.
In view of Corollary \ref{full-residual}, we obtain

\begin{cor}
\label{asymptotically-cyclic}
Every $T\in\L_0(\HH)$ has a nontrivial hyperinvariant subspace.
Hence every asymptotically cyclic, asymptotically nonvanishing a.c.\ contraction has a nontrivial hyperinvariant subspace.
\end{cor}

If $T$ is cyclic, then it is asymptotically cyclic.
If $T$ is not cyclic, then it has nontrivial invariant subspaces. 
In view of Corollary \ref{asymptotically-cyclic} we obtain

\begin{cor}
\label{invariant}
Every asymptotically nonvanishing contraction $T$ has nontrivial invariant subspaces.
\end{cor}

The concept of unitary asymptote can be introduced in connection with any Hilbert space operator; see \cite{ker19} where sufficient conditions are given for its existence.
However, it may happen that $T$ does not have a unitary asymptote even in the case when $T$ is quasisimilar to an isometry; see \cite{Gamal23}.
(The universality property fails.)
If $T$ is a power bounded operator, then the existence of unitary asymptote $(X_T,U_T)$  is ensured by the aid of Banach limits; see \cite{ker89}.

Let us assume now that the operator $T\in\L(\HH)$ is \emph{polynomially bounded}.
Then $T$ can be decomposed into the direct sum $T= T_a \pd T_s$, where $T_a$ is absolutely continuous and $T_s$ is singular in terms of elementary measures; see \cite{ker16}.
The singular component $T_s$ is similar to a singular unitary operator $U_s$, while for $T_a$ there is an $H^\i$-functional calculus.
It has been proved in \cite{Gamal20} that under some conditions $T$ has a shift-type invariant subspace.
Thus, positive answer for Question \ref{main-question} yields the following

\begin{cor}
\label{polynomially-bounded}
If $T$ is a cyclic, a.c.\ polynomially bounded operator and $U_T$ is the simple bilateral shift, then $T$ has a nontrivial hyperinvariant subspace.
\end{cor}

\section{The similarity model}
\label{sim-model}

Let us consider the block-triangular operator
\[T= \left[
\begin{matrix}
A & C\\
0 & B
\end{matrix}\right]
\in\L(\HH_1\op \HH_2),\]
where
$\dim\HH_1=\dim\HH_2=\aleph_0$.
Let us fix an orthonormal basis $\{e_k\}_{k=1}^\i$ in an auxiliary Hilbert space $\HH$, and let us consider the simple unilateral shift $S=\sum_{k=1}^\i e_{k+1}\ot e_k$ on $\HH$.
\emph{We assume that $A$ is similar to $S$}, and so $\HH_1$ is a shift-type invariant subspace of $T$.
Let $V\in\I(S,A)$ be an invertible transformation.

For the operator entry $B$ we apply the technique leading to almost invariant half-spaces; see \cite{HP} and \cite{ker25}.
Select a point on the boundary of the spectrum of $B:\; \m_0\in\partial\s(B)$, and take a sequence $\{\m_k\}_{k=1}^\i$ in the resolvent set $\rho(B)$ of $B$, converging to $\m_0$.
We may assume that $\{\m_k\}_{k=1}^\i$ are distinct, nonzero complex numbers.
If the convergence $\m_k \to \m_0\;(k\to\i)$ is fast enough, then the translate $B-\m_0I$ is similar to an operator
\[\wh B= \left[
\begin{matrix}
* & *\\
F_* & D_*\end{matrix}\right]
\in\L(\HH\op \HH),\]
where $D_*= \sum_{k=1}^\i (\m_k-\m_0) e_k \ot e_k$ and $\hbox{rank}\ F_*=1$.
Let $W\in\I(\wh B,B)$ be an invertible transformation.
Combining these mappings we obtain the invertible transformation
\[Z= \left[\begin{matrix}
V&0\\
0&W\end{matrix}\right]
\in\L\left(\HH\op(\HH\op\HH), \HH_1\op\HH_2\right).\]

\begin{pro}
\label{similarity-model}
We have $(T-\m_0I)Z=Z\wh T$, where
\[\wh T= 
\left[\begin{matrix}
S-\m_0I & V^{-1}CW\\
0& \wh B\end{matrix}
\right] =
\left[\begin{matrix}
S-\m_0I & *&*\\
0&*&*\\
0&F_*&D_*\end{matrix}\right]
\in\L(\HH\op\HH\op\HH).\]
\end{pro}

It is clear that $\hbox{Hlat}\ T = \hbox{Hlat}(T-\m_0I)$ is isomorphic to $\hbox{Hlat}\ \wh T$.

Given any $Q=[Q_{i,j}]_3\in\{\wh T\}'$, comparing the (3,1)-entries in the products $\wh TQ=Q\wh T$, we obtain
\[F_* Q_{2,1} + D_* Q_{3,1} = (\wh T Q)_{3,1} =(Q\wh T)_{3,1}= Q_{3,1}(S-\m_0I),\]
whence
\[(D_*+\m_0I)Q_{3,1} - Q_{3,1}S= -F_*Q_{2,1}.\]
Notice that
\[D:= D_*+\m_0I= \sum_{k=1}^\i \m_k e_k\ot e_k\]
is a cyclic diagonal operator.
We deduce

\begin{pro}
\label{rank1-property}
For every $Q=[Q_{i,j}]_3\in \{\wh T\}'$, we have
\[\hbox{\rm rank}\left(DQ_{3,1} - Q_{3,1}S\right)\le 1.\]
\end{pro}

The entries $Q_{3,1}$ of the operators $Q$ in $\{\wh T\}'$ form the linear manifold $\L_{3,1}$ in $\L(\HH)$.
Transitivity of $\L_{3,1}$ is ensured by a cyclicity property of $T$.

\begin{pro}
\label{cyclicity-assumption}
If the operator $T$ satisfies the condition

\medskip
\noindent
\textup{(H)} $\quad\quad \{T\}'x$ is dense in $\HH_1\op\HH_2$, for every nonzero $x\in\HH_1$,

\medskip
\noindent
then the linear manifold $\L_{3,1}$ is transitive.

 \end{pro}

\noindent
{\bf Proof.}
Given any $0\ne h\in\HH$, we have $0\ne x= Z\wt h\in\HH_1$, where $\wt h= h\op 0\op 0\in \HH^{(3)}$.
In virtue of (H), the linear manifold $\{T\}'x$ is dense in $\HH_1\op \HH_2$.
Hence $Z^{-1}\{T\}'Z\wt h$ is dense in $\HH^{(3)}$, and so $P_3Z^{-1}\{T\}'Z\wt h$ is dense also in $\HH$, where $P_3$ denotes the orthogonal projection of $\HH^{(3)}$ onto $\{0\}\op\{0\}\op\HH$.
But $Z^{-1}\{T\}'Z= \{\wh T\}'$, and for any $Q=[Q_{i,j}]_3\in\{\wh T\}'$ we have $P_3Q\wt h= Q_{3,1}h$.

\rightline{$\square$}

We are going to examine the relation of the properties in the previous two propositions by considering 3-dimensional cross-sections of $\L_{3,1}$.

\section{Cross-sections of the commutant}
\label{cross-sections}

For any $r\in \N$, let us consider the compressions of the arising operators to the 3-dimensional subspace $\E_{3,r}= \vee\{e_r, e_{r+1}, e_{r+2}\}$.
Namely
\[\D_r := D|_{\E_{3,r}} = \sum_{i=1}^3 \m_{r+i-1} (e_{r+i-1}\ot e_{r+i-1})\]
and
\[S_r:= P_{\E_{3,r}}S|_{\E_{3,r}}= e_{r+1}\ot e_r+ e_{r+2}\ot e_{r+1}.\]
Let $\L_{*,r}$ stand for the set of compressions of the operators in $\L_{3,1}$ to $\E_{3,r}$:
\[\L_{*,r}= \left\{ P_{\E_{3,r}}\wt X|_{\E_{3,r}}: \wt X \in \L_{3,1}\right\}.\]
Clearly, $\L_{*,r}$ is a subspace in $\L(\E_{3,r})$.
From Proposition \ref{rank1-property} we may derive the following statement.

\begin{pro}
\label{rank2-condition}
For every $X\in\L_{*,r}$, we have
\[\hbox{\rm rank}(\D_r X-X S_r)\le 2.\]
\end{pro}

\noindent
{\bf Proof.}
Given any $X= P_{\E_{3,r}}Q_{3,1}|_{\E_{3,r}}$, with $Q=[Q_{i,j}]_3\in\{\wh T\}'$, we have
\[Z:= P_{\E_{3,r}}\left[DQ_{3,1} - Q_{3,1} S\right]|_{\E_{3,r}} = \D_r X-XS_r -Y,\]
where $Y= P_{\E_{3,r}} Q_{3,1} (I-P_{\E_{3,r}})S|_{\E_{3,r}}$ and $\hbox{rank}\ Y\le 1$.
In view of Proposition \ref{rank1-property}, we obtain
\[\hbox{rank}(\D_r X- X S_r) = \hbox{rank}(Z+Y)\le 2.\]

\rightline{$\square$}

Furthermore, in virtue of Proposition \ref{cyclicity-assumption}, we readily obtain

\begin{pro}
\label{cross-section-transitivity}
If $T$ satisfies the condition $\hbox{\rm (H)}$, then the cross-section $\L_{*,r}$ is a transitive subspace in $\L(\E_{3,r})$.
\end{pro}

We notice that in finite dimension transitivity means strict transitivity, that is $\L_{*,r} x = \E_{3,r}$ holds for every $0\ne x\in\E_{3,r}$.

\section{Subspaces of $3 \times 3$ complex matrices}
\label{subspaces-of-matrices}

Fixing an orthonormal basis $(e_1,e_2,e_3)$ in the Hilbert space $\E_3$, the algebra $\L(\E_3)$ of operators acting on $\E_3$ can be identified with the algebra $M_3[\C]$ of $3 \times 3$ complex matrices.
Let $\DD_3$ stand for the system of cyclic diagonal matrices 
$\D = \hbox{diag}(b_1,b_2,b_3)$ in $M_3[\C]$; thus $b_1, b_2, b_3$ are distinct nonzero complex numbers.
Let us take also the Jordan-cell
\[J= \left[\begin{matrix}
0&0&0\\
1&0&0\\
0&1&0\end{matrix}\right]
\in M_3[\C].\]
Given any $\D= \hbox{diag}(b_1,b_2,b_3)\in\DD_3$, we consider the linear mapping
\[\G_\D \co M_3[\C] \to M_3[\C], \quad X \mapsto \D X -X J.\]
Since $\s(\D)\cap\s(J)= \{b_1,b_2,b_3\}\cap\{0\} =\emptyset$, it follows by Rosenblum's theorem that $\G_\D$ is invertible.

Let $\wt\F_2$ denote the variety of singular matrices; that is $X\in\wt\F_2$ if $\hbox{rank}\ X\le 2$, or equivalently $\det X=0$.
Let $\hbox{Lat} M_3[\C]$ denote the lattice of subspaces in the vector space $M_3[\C]$.
Given any $\D\in\DD_3$, the subspace $\L\in \hbox{Lat}M_3[\C]$ is called \emph{$\D$-singular}, in notation: $\L\in \S_3[\D]$, if
\[\G_\D(\L):= \{\G_\D(X): X\in\L\} \subset \wt\F_2.\]
On the other hand, $\L$ is \emph{$\D$-regular}, in notation: $\L\in\R_3[\D]$, if there exists $X\in\L$ such that $\det \G_\D(X)\ne 0$.
The subspaces in
\[\R_3^+= \cap\{\R_3[\D]: \D\in\DD_3\}\]
are called \emph{strongly regular}.
Finally, let $\T_3$ denote the system of \emph{transitive subspaces} in $\hbox{Lat}M_3[\C]$.
Our main goal is to detect the connection of these properties.
The next statement describes the possible dimensions.

\begin{pro}
\label{possible-dimension}
Let $\L\in \hbox{\rm Lat}M_3[\C]$ be given.
\begin{itemize}
\item[\textup{(a)}]  If $\L\in\S_3[\D]$ holds with some $\D\in\DD_3$, then $\dim\L\le6$.
\item[\textup{(b)}] If $\L\in\T_3$, then $\dim\L\ge5$.
\end{itemize}
\end{pro}

\noindent
{\bf Proof.} If $\L\in\S_3[\D]$, then the subspace $\G_\D(\L):= \{\G_\D(X): X\in\L\}$ is included in $\wt\F_2$, and so $\dim\G_\D(\L)\le 6$ by \cite{Flanders}.
Furthermore, the invertibility of $\G_\D$ yields that $\dim\L= \dim\G_\D(\L)$.
In connection with (b), we refer to \cite{Azoff}.

\rightline{$\square$}

Given any $X=[\x_{i,j}]_3\in M_3[\C]$ and $\D= \hbox{diag}(b_1,b_2,b_3)\in\DD_3$, we have
\[\G_\D(X)=
\left[\begin{matrix}
b_1\x_{1,1}-\x_{1,2} & b_1\x_{1,2}-\x_{1,3} & b_1\x_{1,3}\\

b_2 \x_{2,1}- \x_{2,2} & b_2 \x_{2,2}- \x_{2,3} & b_2\x_{2,3}\\

b_3 \x_{3,1}-\x_{3,2} & b_3\x_{3,2}-\x_{3,3} & b_3\x_{3,3}
\end{matrix}\right].\]
The determinant of this matrix can be easily evaluated if there are two zeros in the third column of $X$.
For any $r\in\N_3:=\{1,2,3\}$, let $P_r:=(k_1,k_2)$, where $k_1<k_2$ and $\{k_1,k_2\}\cup\{r\}=\N_3$.
Let $M_{P_r}$ denote the system of those matrices $X= [\x_{i,j}]_3$, where $\x_{k_1,3}=\x_{k_2,3}=0$.
The $2 \times 2$ submatrix below is of great significance
\[X_{P_r}= \left[\begin{matrix}
\x_{k_1,1} & \x_{k_1,2}\\
\x_{k_2,1} & \x_{k_2,2}
\end{matrix}\right].\]
Expanding the determinant along the third column, we obtain
\[\det\G_\D(X)= (-1)^{r+3} b_1b_2b_3\ \x_{r,3} \left(\det X_{P_r} + \left(\frac{1}{b_{k_2}}-\frac{1}{b_{k_1}}\right) \x_{k_1,2}\x_{k_2,2}\right).\]
For any $b\in\C_0:= \C\setminus\{0\}$, let us consider the functional
\[F_b\co M_2[\C] \to \C, \quad Z=[\z_{i,j}]_2 \mapsto \det Z + b\  \z_{1,2}\z_{2,2}.\]
Let us introduce also the notation
\[b_\D(r):= \frac{1}{b_{k_2}}-\frac{1}{b_{k_1}} (\ne 0).\]
As a result we obtain

\begin{pro}
\label{gamma-determinant}
If $X=[\x_{i,j}]_3\in M_{P_r}\; (r\in\N_3)$ and $\D= \hbox{\rm diag}(b_1,b_2,b_3)\in\DD_3$, then
\[\det\G_\D(X)= (-1)^{r+3} b_1b_2b_3 \ \x_{r,3}\ F_{b_\D(r)}(X_{P_r}).\]
\end{pro}

Let $\M_3$ denote the subspace of those matrices $X=[\x_{i,j}]_3$, where $\x_{1,3} = \x_{2,3}= \x_{3,3}=0$.
For any $\L\in \hbox{Lat}M_3[\C]$, let us consider the subspace \[\L_0:=\L\cap\M_3.\]

We say that $\L_0$ is \emph{functionally $\D$-vanishing}, if for every $r\in\N_3$ and for all $L\in\L_0$, we have $F_{b_\D(r)}(L_{P_r})=0$.

\begin{pro}
\label{regularity-condition}
If $\L\in\T_3$ and $\L_0$ is not functionally $\D$-vanishing for some $\D\in\DD_3$, then $\L\in\R_3[\D]$.
\end{pro}

\noindent
{\bf Proof.} Let
$\D= \hbox{diag}(b_1,b_2,b_3)\in\DD_3$ be arbitrary, and suppose that $\L_0$ is not functionally $\D$-vanishing.
Then there exist $r\in\N_3$ and $L\in\L_0$ such that $F_{b_\D(r)}(L_{P_r})\ne 0$.
By the transitivity of $\L$ there exists $Y=[\eta_{i,j}]_3\in\L$ such that $Ye_3=e_r$, that is $\eta_{i,3}= \d_{i,r}\; (i\in\N_3)$.
Clearly, $X=zL+Y$ belongs to $M_{P_r}\cap\L$ for every $z\in\C$, and so by Proposition \ref{gamma-determinant} we have
\[\det\G_\D(X)= (-1)^{r+3} b_1b_2b_3\ F_{b_\D(r)}(X_{P_r}).\]
Since
\[X_{P_r}= (zL)_{P_r} + Y_{P_r},\]
it is easy to verify that
\[F_{b_\D(r)}(X_{P_r}) = z^2 F_{b_\D(r)}(L_{P_r})+ cz + F_{b_\D(r)}(Y_{P_r})=: p(z),\]
with some $c\in\C$.
Taking into account that $F_{b_\D(r)}(L_{P_r})\ne0$, it follows that the polynomial $p(z)$ is nonzero for all $z\in\C$, with at most two exceptions.
Thus $\det\G_\D(X)\ne0$ is true with many choices of $z$.

\rightline{$\square$}

The subspace $\L_0$ is called \emph{functionally nonvanishing}, if there is an $r\in\N_3$ such that, for any $b\in\C_0, \; F_b(L_{P_r})\ne0$ holds for some $L\in\L_0$.
Clearly, if $\L_0$ is functionally nonvanishing, then $\L_0$ is not functionally $\D$-vanishing for any $\D\in\DD_3$, and so Proposition \ref{regularity-condition} can be applied.

\begin{cor}
\label{strong-regularity-condition}
If $\L\in\T_3$ and $\L_0$ is functionally nonvanishing, then $\L\in\R_3^+$.
\end{cor}

We are going to examine whether $\D$-singularity and transitivity are consistent properties.
In view of Proposition \ref{possible-dimension} we may restrict our attention to 5- and 6-dimensional subspaces.
First we study the 6-dimensional case.

\section{6-dimensional transitive subspaces}
\label{6dimensional-transitive}

\emph{Let us be given a subspace $\L\in\T_3$ with  $\dim\L=6$, and let us consider $\L_0=\L\cap\M_3$.}

\begin{pro}
\label{3dimensional-L0}
The subspace $\L_0$ is $3$-dimensional.
\end{pro}

\noindent
{\bf Proof.}
Since $\L$ is transitive, the linear mapping $\Lambda\co \L\to\E_3, X\mapsto Xe_3$ is surjective.
Hence, for its nullspace $\ker\Lambda=\L_0$, we have $\dim\L_0= \dim\L - \dim\E_3= 6-3=3$.

\rightline{$\square$}

Let us consider the lexicographic ordering on $ \N_3 \times \N_2:\; (i_1,j_1)\prec (i_2,j_2)$ if $i_1<i_2$, or $i_1=i_2$ and $j_1<j_2$.
Let $\J_3$ stand for the system of increasing triplets of pairs in $\N_3\times\N_2$.
We may introduce a canonical basis in $\L_0$, indexed by the elements of $\J_3$.

\begin{lem}
\label{3dimensional-L0-basis}
There exist unique triplet $J= \left((i_1,j_1), (i_2,j_2), (i_3,j_3)\right)\in \J_3$ and matrices $L_k = [l_{i,j}^k]_3$ in $\L_0\; (k\in\N_3)$ such that
\[l_{i_r,j_r}^k = \d_{k,r}\quad (k,r\in\N_3)\]
and
\[l_{i,j}^k=0\quad \hbox{if} \;\; (i,j)\prec (i_k,j_k)\; \; (k\in\N_3).\]
\end{lem}

For the proof, we refer to Lemma 8 in \cite{ker25}.
We say that \emph{the basis $(L_1,L_2,L_3)$ and the subspace $\L_0$ itself is of type $\J$.}

We are going to show that $\L_0$ is functionally nonvanishing with some special exceptions.
For any $(k,l)\in\N_3\times\N_2$, let $E_{k,l}= [e_{i,j}^{k,l}]_3$ be the elementary matrix, where $e_{i,j}^{k,l} = \d_{(i,j), (k,l)}.$
(These matrices form the natural basis in $M_3[\C]$.)

Our method is the following.
The general element in $\L_0$ is of the form
\[L= z_1 L_1 +z_2L_2+z_3L_3 \quad (z_1,z_2,z_3\in\C).\]
Given any $r\in\N_3$, for every $b\in\C_0$ we consider the polynomial
\[F_b(L_{P_r})=: p_b(z_1,z_2,z_3).\]
Taking the canonical form
\[p_b(z_1,z_2,z_3)= \sum a_{k_1,k_2,k_3}(b)\ z_1^{k_1} z_2^{k_2} z_3^{k_3},\]
this polynomial is identically zero if and only if all coefficients $a_{k_1,k_2,k_3}(b)$ are equal to zero.
(Here $(k_1,k_2,k_3)$ runs through the triplets of nonnegative integers: $\Z_+^3$.)

\begin{lem}
\label{coefficient-condition}
The subspace $\L_0$ is functionally nonvanishing exactly when

\medskip
\noindent
\textup{(*)} $\quad \exists r\in\N_3, \; \forall b\in\C_0, \; \exists (k_1,k_2,k_3)\in \Z_+^3, \; a_{k_1,k_2,k_3}(b)\ne 0.$
\end{lem}

\noindent
In order to carry out a systematic study we classify the indeces $J\in\J_3$:
\begin{itemize}
\item[\textup{(I)}] $\quad \hbox{card}\{i_1,i_2,i_3\}=2 \quad$ (12 cases),
\item[\textup{(II)}] $\quad (i_1,i_2,i_3)=(1,2,3)$ and $\hbox{card}\{k\in\N_3: j_k=1\}=2\quad$ (3 cases),
\item[\textup{(III)}] $\quad (i_1,i_2,i_3)= (1,2,3)$ and $\hbox{card}\{k\in\N_3: j_k=2\}=2 \quad$ (3 cases),
\item[\textup{(IV)}] $\quad (i_1,i_2,i_3)= (1,2,3)$ and $\hbox{card}\{j_1,j_2,j_3\}=1 \quad$ (2 cases).
\end{itemize}

\begin{pro}
\label{classI-3basis}
The subspace $\L_0$ is functionally nonvanishing whenever the type of $\L_0$ is of class \textup{(I)}.
\end{pro}

\noindent
{\bf Proof.}
Let us assume that $\L_0$ is of type $((1,1),(1,2),(2,1))$.
Then
\[L_1=
\left[
\begin{matrix}
1&0&0\\
0&l_{2,2}^1 & 0\\
*&*&0
\end{matrix}\right], \quad
L_2= \left[
\begin{matrix}
0&1&0\\0& l_{2,2}^2&0\\
*&*&0
\end{matrix}\right], \quad
L_3= \left[\begin{matrix}
0&0&0\\
1&l_{2,2}^3&0\\
*&*&0
\end{matrix}\right],\]
and so
\[L=z_1L_1+z_2L_2+z_3L_3=
\left[\begin{matrix}
z_1&z_2&0\\
z_3& z_1l_{2,2}^1+z_2 l_{2,2}^2+ z_3 l_{2,2}^3& 0\\
*&*&0
\end{matrix}\right].\]
Taking $P_3=(1,2)$, for every $b\in\C_0$, we have
\begin{eqnarray*}
F_b(L_{P_3}) & = & z_1(z_1l_{2,2}^1 + z_2 l_{2,2}^2 + z_3 l_{2,2}^3) - z_2z_3 + bz_2 (z_1 l_{2,2}^1 +z_2 l_{2,2}^2 +z_3 l_{2,2}^3) \\
 & = & z_1^2 l_{2,2}^1 + z_1z_2 (l_{2,2}^2 +b l_{2,2}^1)+ z_1z_3 l_{2,2}^3 + z_2^2 b l_{2,2}^2 + z_2z_3 (-1+bl_{2,2}^3)
\end{eqnarray*}
Since $a_{1,0,1}= l_{2,2}^3$ and $a_{0,1,1} =-1 +b l_{2,2}^3$ cannot be simultaneously zero, there is a nonzero coefficient.

Let us assume now that $\L_0$ is of type $((1,1), (1,2), (2,2))$.
Then
\[L_1= \left[
\begin{matrix}
1&0&0\\
l_{2,1}^1&0&0\\
*&*&0
\end{matrix}\right], \quad L_2= \left[\begin{matrix}
0&1&0\\
l_{2,1}^2&0&0\\
*&*&0
\end{matrix}\right], \quad
L_3= \left[\begin{matrix}
0&0&0\\
0&1&0\\
*&*&0
\end{matrix}\right],\]
and so
\[L=z_1L_1+z_2L_2+z_3L_3 = 
\left[\begin{matrix}
z_1 & z_2 & 0\\
z_1 l_{2,1}^1 + z_2 l_{2,1}^2& z_3&0\\
*&*&0
\end{matrix}\right].\]
Taking $P_3=(1,2)$, for every $b\in\C_0$ we have 
\begin{eqnarray*}
F_b(L_{P_3})&=& z_1z_3 - z_2 (z_1 l_{2,1}^1 + z_2 l_{2,1}^2) + b z_2z_3 \\
 &=&-z_1z_2l_{2,1}^1 + z_1z_3 - z_2^2 l_{2,1}^2 +bz_2z_3,
\end{eqnarray*}
which clearly has nonzero coefficients.

Let us assume that $\L_0$ is of type $((1,1),(2,1),(2,2))$.
Then
\[L_1=
\left[\begin{matrix}
1&l_{1,2}^1&0\\
0&0&0\\
*&*&0
\end{matrix}\right], \quad
L_2= \left[\begin{matrix}
0&0&0\\
1&0&0\\
*&*&0\end{matrix}\right], \quad
L_3= \left[\begin{matrix}
0&0&0\\
0&1&0\\
*&*&0
\end{matrix}\right],\]
and so
\[L=\left[\begin{matrix}
z_1& z_1l_{1,2}^1&0\\
z_2&z_3&0\\
*&*&0
\end{matrix}\right].\]
For every $b\in\C_0$, we have
\[F_b(L_{P_3})= z_1z_3 - z_1z_2 l_{1,2}^1 + b l_{1,2}^1 z_1z_3 = z_1z_3(1+bl_{1,2}^1)-z_1z_2l_{1,2}^1.\]
Since $1+bl_{1,2}^1$ and $l_{1,2}^1$ cannot be simultaneously zero, there is a nonzero coefficient.

Let us assume that $\L_0$ is of type $((1,2),(2,1),(2,2))$.
Then
\[L_1= \left[\begin{matrix}
0&1&0\\
0&0&0\\
*&*&0
\end{matrix}\right], \quad
L_2=
\left[\begin{matrix}
0&0&0\\
1&0&0\\
*&*&0
\end{matrix}\right], \quad
L_3= \left[\begin{matrix}
0&0&0\\
0&1&0\\
*&*&0
\end{matrix}\right],\]
and so
\[L=\left[\begin{matrix}
0&z_1&0\\
z_2&z_3&0\\
*&*&0\end{matrix}\right].\]
For every $b\in\C_0$, we have
\[F_b(L_{P_3})= -z_1z_2 + b z_1z_3,\]
which clearly has nonzero coefficient.

The other cases when $\{i_1,i_2,i_3\}=\{1,3\}$ or $\{i_1,i_2,i_3\}=\{2,3\}$ can be handled similarly, taking $r=2$ or $r=1$, respectively.

\rightline{$\square$}

In the class (II) there are exceptions.

\begin{pro}
\label{classII-3basis}
If the type $J$ of $\L_0$ is of class \textup{(II)}, then $\L_0$ is functionally nonvanishing, with the exceptions listed below.

\medskip
\textup{(a)} If $J=((1,1), (2,1),(3,2))$, then

\medskip
\noindent
\textup{(E1)} $\quad L_1= E_{1,1} + c_1 E_{1,2}, \quad L_2= E_{2,1} + c_2 E_{2,2}, \quad L_3 = E_{3,2} \quad (c_1,c_2\in\C_0).$

\medskip
\textup{(b)} If $J= ((1,1), (2,2), (3,1))$, then

\medskip
\noindent
\textup{(E2)} $\quad L_1= E_{1,1} + c_1 E_{1,2}, \quad L_2= E_{2,2}, \quad L_3= E_{3,1} + c_3 E_{3,2}\quad (c_1, c_3\in\C_0).$

\medskip
\textup{(c)}
If $J=((1,2),(2,1),(3,1))$, then

\medskip
\noindent
\textup{(E3)} $\quad L_1= E_{1,2}, \quad L_2= E_{2,1} + c_2 E_{2,2}, \quad L_3 = E_{3,1} + c_3 E_{3,2}\quad (c_2, c_3\in\C_0).$
\end{pro}

\noindent
{\bf Proof.}

(a): In that case
\[L_1= \left[\begin{matrix}
1&l_{1,2}^1&0\\
0&l_{2,2}^1&0\\
l_{3,1}^1&0&0
\end{matrix}\right], \quad
L_2= \left[\begin{matrix}
0&0&0\\
1&l_{2,2}^2&0\\
l_{3,1}^2&0&0\end{matrix}\right], \quad
L_3= \left[\begin{matrix}
0&0&0\\
0&0&0\\
0&1&0\end{matrix}\right],\]
and so
\[L=\left[\begin{matrix}
z_1& z_1l_{1,2}^1&0\\
z_2& z_1l_{2,2}^1+z_2 l_{2,2}^2&0\\
z_1l_{3,1}^1+z_2 l_{3,1}^2&z_3&0\end{matrix}\right].\]
For any $b\in\C_0$, we have
\begin{eqnarray*}
F_b(L_{P_1}) & = & z_2z_3 - (z_1l_{2,2}^1+ z_2 l_{2,2}^2) ( z_1 l_{3,1}^1+z_2 l_{3,1}^2) + b (z_1 l_{2,2}^1 + z_2 l_{2,2}^2) z_3\\
 &=& -z_1^2 l_{2,2}^1 l_{3,1}^1 - z_1z_2 (l_{2,2}^1 l_{3,1}^2 + l_{2,2}^2 l_{3,1}^1) + z_1z_3 l_{2,2}^1 b\\
 & & -l_{2,2}^2 l_{3,1}^2 z_2^2 + z_2z_3 (1+ b l_{2,2}^2)\\
 &=&  p_b(z_1,z_2,z_3).
\end{eqnarray*}
For every $b\in\C_0$, this polynomial has a nonzero coefficient if $l_{2,2}^1\ne0$ or $l_{2,2}^2=0$.
Hence we may assume that $l_{2,2}^1=0$ and $l_{2,2}^2\ne0$.
Under these conditions
\[p_b(z_1,z_2,z_3)= - z_1z_2 l_{2,2}^2 l_{3,1}^1 - z_2^2 l_{2,2}^2 l_{3,1}^2 + z_2z_3(1+ b l_{2,2}^2).\]
There is a nonzero coefficient here for every $b\in\C_0$, if $l_{3,1}^1\ne 0$ or $l_{3,1}^2\ne 0$.
So we may assume that $l_{3,1}^1= l_{3,1}^2=0$.
Then
\[F_b(L_{P_2})= z_1z_3 (1+ b l_{1,2}^1).\]
The arising coefficient is nonzero for all $b\in\C_0$, if $l_{1,2}^1=0$.
Hence, we may assume $l_{1,2}^1\ne0$ in order to exclude validity of the condition (*), and we arrive at the form (E1).

The situations (b) and (c) can be settled analogously.

\rightline{$\square$}

In the class (III) there are no exceptions.

\begin{pro}
\label{classIII-3basis}
The subspace $\L_0$ is functionally nonvanishing whenever the type of $\L_0$ is of class \textup{(III)}.
\end{pro}

\noindent
{\bf Proof.}
Suppose that $\L_0$ is of type $((1,1),(2,2),(3,2))$.
Then
\[L_1=
\left[\begin{matrix}
1&l_{1,2}^1&0\\
l_{2,1}^1&0&0\\
l_{3,1}^1&0&0\end{matrix}\right], \quad
L_2= \left[\begin{matrix}
0&0&0\\
0&1&0\\
l_{3,1}^2&0&0\end{matrix}\right], \quad
L_3= \left[\begin{matrix}
0&0&0\\
0&0&0\\
0&1&0\end{matrix}\right],\]
and so
\[L=z_1L_1+z_2L_2+z_3L_3= \left[\begin{matrix}
z_1&z_1l_{1,2}^1&0\\
z_1l_{2,1}^1&z_2&0\\
z_1l_{3,1}^1+z_2l_{3,1}^2&z_3&0\end{matrix}\right].\]
For any $b\in\C_0$, we have
\begin{eqnarray*}
F_b(L_{P_1}) & = & z_1z_3 l_{2,1}^1 - z_2(z_1l_{3,1}^1+ z_2 l_{3,1}^2)+ b z_2z_3\\
  & = & -z_1z_2 l_{3,1}^1 + z_1z_3 l_{2,1}^1 - z_2^2 l_{3,1}^2 + b z_2z_3.
\end{eqnarray*}
Since $a_{0,1,1}(b)=b\ne0$, it follows that $\L_0$ is functionally nonvanishing.

Similarly, if the type of $\L_0$ is $((1,2),(2,1),(3,2))$, then for $r=2$ and for every $b\in\C_0$, the coefficient $a_{1,0,1}(b)=b\ne0$ in $F_b(L_{P_2})= p_b(z_1,z_2,z_3)$.
Finally, if the type of $\L_0$ is $((1,2),(2,2),(3,1))$, then for $r=3$ and for every $b\in\C_0$, the coefficient $a_{1,1,0}(b)=b\ne 0$ in $F_b(L_{P_3})= p_b(z_1,z_2,z_3)$.

\rightline{$\square$}

In the class (IV) there are again exceptions.

\begin{pro}
\label{classIV-3basis}
If the type $J$ of $\L_0$ is of class \textup{(IV)}, then $\L_0$ is functionally nonvanishing with the exceptions when $J=((1,1),(2,1),(3,1))$ and

\medskip
\noindent
\textup{(E4)} $\quad L_1=E_{1,1}, \quad L_2= E_{2,1}, \quad L_3=E_{3,1}$, or

\medskip
\noindent
\textup{(E5)} $\quad L_1= E_{1,1}+ c_1E_{1,2}, \quad L_2= E_{2,1}+ c_2E_{2,2}, \quad L_3= E_{3,1}+c_3E_{3,2}$

 $\quad\;\;\hbox{ with distinct}\; c_1,c_2,c_3\in\C_0$.
\end{pro}

In the proof we shall need the following technical lemma.

\begin{lem}
\label{simultaneous-2equations}
Given $l_1, l_2,l_3\in\C$, let us consider the functions
\[a_1(b)= l_3(1+bl_1) \quad \hbox{and} \quad a_2(b)= l_2-l_1+ b l_1l_2 \quad (b\in\C).\]
There exists $b\in\C_0$ such that $a_1(b)=a_2(b)=0$ if and only if either $l_1=l_2=l_3=0$ or
\[l_1, l_2\in\C_0, \; l_1\ne l_2\; \hbox{and}\; l_3=0.\]

The same statement holds also if $a_1(b)$ is replaced by $\wt a_1(b)= l_3(-1+bl_2)$.
\end{lem}

\noindent
{\bf Proof.}
If $l_1=0$, then $a_1(b)=l_3$ and $a_2(b)= l_2$, and so the first option is valid.

Suppose that $l_1\ne 0$.
If $l_2=0$, then $a_2(b)=-l_1\ne0$; so we may assume that $l_2\ne 0$.
Then $l_2=l_1$ implies $a_2(b)= b l_1 l_2\ne0$; so $l_1\ne l_2$ can be assumed.
Now $a_2(b)=0$ holds exactly when
\[b= \frac{l_1-l_2}{l_1l_2}.\]
Hence
\[a_1(b)= l_3(1+bl_1)= l_3 \left(1 + \frac{l_1-l_2}{l_1l_2} l_1\right)= l_3 \frac{l_1}{l_2}\]
is zero precisely when $l_3=0$.

The  situation with $\wt a_1(b)$ can be settled similarly.

\rightline{$\square$}

\noindent
{\bf Proof of Proposition \ref{classIV-3basis}.}

The type $((1,2),(2,2),(3,2))$ can be easily handled.
Indeed
\[L_1= \left[\begin{matrix}
0&1&0\\
l_{2,1}^1&0&0\\
l_{3,1}^1&0&0\end{matrix}\right], \quad
L_2= \left[\begin{matrix}
0&0&0\\
0&1&0\\
l_{3,1}^2&0&0\end{matrix}\right], \quad L_3= \left[\begin{matrix}
0&0&0\\
0&0&0\\
0&1&0\end{matrix}\right],\]
and so
\[L= z_1L_1+z_2L_2+z_3L_3 =
\left[\begin{matrix}
0&z_1&0\\
z_1 l_{2,1}^1&z_2&0\\
z_1l_{3,1}^1+z_2l_{3,1}^2&z_3&0\end{matrix}\right].\]
For $P_3=(1,2)$ and for every $b\in\C_0$, we have
\[F_b(L_{P_3})= - z_1^2 l_{2,1}^1 + b z_1z_2.\]
Since $a_{1,1,0}(b)=b\ne 0$, it follows that $\L_0$ is functionally nonvanishing.

Let us assume now that $\L_0$ is of type $((1,1), (2,1), (3,1))$.
Then
\[L_1= \left[\begin{matrix}
1& l_{1,2}^1 &0\\
0& l_{2,2}^1 & 0\\
0& l_{3,2}^1 &0\end{matrix}\right], \quad
L_2= \left[\begin{matrix}
0&0&0\\
1& l_{2,2}^2 &0\\
0& l_{3,2}^2&0\end{matrix}\right], \quad
L_3= \left[\begin{matrix} 
0&0&0\\
0&0&0\\
1& l_{3,2}^3 & 0\end{matrix}\right],\]
and so
\[L=z_1L_1+z_2L_2+z_3L_3= \left[\begin{matrix}
z_1& z_1 l_{1,2}^1 &0\\
z_2 & z_1 l_{2,2}^1+z_2 l_{2,2}^2& 0\\
z_3& z_1 l_{3,2}^1 + z_2 l_{3,2}^2 + z_3 l_{3,2}^3 &0\end{matrix}\right].\]
Let us use the notation
\[F_b(L_{P_r})= p_{r,b}(z_1,z_2,z_3)= \sum a_{k_1,k_2,k_3}(r,b) z_1^{k_1}z_2^{k_2}z_3^{k_3}\quad (r\in \N_3, b\in\C_0).\]
Then
\begin{eqnarray*}
p_{3,b}(z_1,z_2,z_3) & = & z_1^2 l_{2,2}^1 (1+ b l_{1,2}^1) + z_1z_2 (l_{2,2}^2 - l_{1,2}^1 + b l_{1,2}^1 l_{2,2}^2),\\
p_{2,b}(z_1,z_2,z_3) & = & z_1^2 l_{3,2}^1 (1 + b l_{1,2}^1)\\
 & &  + z_1z_2 l_{3,2}^2 (1+ b l_{1,2}^1) + z_1z_3 (l_{3,2}^3 - l_{1,2}^1 + b l_{1,2}^1 l_{3,2}^3),\\
p_{1,b}(z_1,z_2,z_3) & = & z_1^2 b l_{2,2}^1 l_{3,2}^1 + z_1z_2 (l_{3,2}^1 + b l_{2,2}^1 l_{3,2}^2 + b l_{2,2}^2 l_{3,2}^1)\\
 & & + z_1z_3 l_{2,2}^1(-1 + b l_{3,2}^3) + z_2^2 l_{3,2}^2(1+ b l_{2,2}^2)\\
 & & + z_2z_3 (l_{3,2}^3 - l_{2,2}^2 + b l_{2,2}^2 l_{3,2}^3).
\end{eqnarray*}

Let us assume that $\L_0$ is not functionally nonvanishing.
In view of Lemma \ref{coefficient-condition} this means that for every $r\in\N_3$, there exists $b\in\C_0$ such that $a_{k_1,k_2,k_3}(r,b)=0$ is true for all $(k_1,k_2,k_3)\in \Z_+^3$.
We will apply Lemma \ref{simultaneous-2equations} taking the following substitutions:  
\begin{itemize}
\item[\textup{(S1)}] $\quad l_1= l_{1,2}^1, \; l_2= l_{2,2}^2, \; l_3= l_{2,2}^1; \; a_1(b)= a_{2,0,0}(3,b), \; a_2(b)= a_{1,1,0}(3,b);$
\item[\textup{(S2)}] $\quad l_1= l_{1,2}^1, \; l_2 = l_{3,2}^3, \; l_3= l_{3,2}^1; \; a_1(b)= a_{2,0,0}(2,b), \; a_2(b)= a_{1,0,1}(2,b);$
\item[\textup{(S3)}] $\quad l_1=l_{1,2}^1, \; l_2 = l_{3,2}^3, \; l_3= l_{3,2}^2; \; a_1(b)= a_{1,1,0}(2,b), \; a_2(b)= a_{1,0,1}(2,b)$.
\end{itemize}
Let us assume first that $l_{1,2}^1=0$.
Then by Lemma \ref{simultaneous-2equations},
\begin{itemize}
\item[{}] (S1) yields $\; l_{2,2}^2=l_{2,2}^1=0$,
\item[{}] (S2) yields $\; l_{3,2}^3=l_{3,2}^1=0$,
\item[{}] (S3) yields $\; l_{3,2}^3=l_{3,2}^2=0$.
\end{itemize}
Therefore, we obtain the form (E4).

\noindent
Let us assume now that $l_{1,2}^1\ne0$.
Then by Lemma \ref{simultaneous-2equations},
\begin{itemize}
\item[{}] (S1) yields $\; 0\ne l_{2,2}^2\ne l_{1,2}^1$ and $l_{2,2}^1=0$,
\item[{}] (S2) yields $\; 0\ne l_{3,2}^3\ne l_{1,2}^1$ and $l_{3,2}^1=0$,
\item[{}] (S3) yields $\; 0\ne l_{3,2}^3\ne l_{1,2}^1$ and $l_{3,2}^2=0$.
\end{itemize}
Taking into account that
\[a_{0,1,1}(1,b)= l_{3,2}^3 - l_{2,2}^2 + b l_{2,2}^2 l_{3,2}^3 = 0\]
holds with some $b\in\C_0$, we infer that $l_{3,2}^3\ne l_{2,2}^2$ must be also true.
Thus, we have obtained the form (E5).

\rightline{$\square$}

In order to study the exceptional cases we complete the basis $(L_1,L_2,L_3)$ of $\L_0$ to a basis of $\L$.

\begin{lem}
\label{6dimension-extended-basis}
Let $(L_1,L_2,L_3)$ be a canonical basis in $\L_0$ of type 
\[J=((i_1,j_1), (i_2,j_2),(i_3,j_3))\in\J_3.\]
Then there exist matrices $Q_r = [q_{i,j}^r]_3\in\L \;\; (r\in\N_3)$ such that 
\[q_{i,3}^r= \d_{i,r}\quad (i,r\in \N_3)\]
and
\[q_{i_k,j_k}^r = 0\quad (k,r\in\N_3).\]
\end{lem}

For the proof we refer to Lemma 10 in \cite{ker25}.
The uniquely determined system $(L_1,L_2,L_3, Q_1,Q_2,Q_3)$ is called the \emph{canonical basis} in $\L$.

We enlarge the subspace $\L_0$ by one dimension.
Namely, given any $r\in\N_3$, let $\L_r$ denote the 4-dimensional subspace spanned by $\L_0$ and $Q_r$.
The general element of $\L_r$ is of the form
\[X= z_1L_1 + z_2L_2 +z_3L_3 + w_rQ_r\quad (z_1, z_2, z_3, w_r\in\C).\]
Since $\L_r = \L\cap M_{P_r}\subset M_{P_r}$, the formula in Proposition \ref{gamma-determinant} can be applied.
Hence, given any $\D= \hbox{diag}(b_1,b_2,b_3)\in\DD_3$, we have
\[\det\G_\D(X)= (-1)^{r+3} b_1b_2b_3 w_r\ F_{b_\D(r)}(X_{P_r}).\]
Given any $b\in\C_0$, we say that $\L_r$ is \emph{functionally $b$-vanishing}, if $F_b(X_{P_r})=0$ holds for every $X\in\L_r$.
On the other hand, $\L_r$ is called \emph{functionally nonvanishing}, if for every $b\in\C_0$ there exists $X\in\L_r$ such that $F_b(X_{P_r})\ne 0$.
Notice also that the polynomial $F_b(X_{P_r})= p_b(z_1,z_2,z_3,w_r)$ is not identically zero exactly when it has a nonzero coefficient.
Evidently, $p_b(z_1,z_2,z_3,w_r)$ and $w_r p_b(z_1,z_2,z_3,w_r)$ are identically zero at the same time.
Thus, in view of the previous formula we obtain the next statement.

\begin{pro}
\label{functional-conditions-6dim}
Let $r\in\N_3$ be arbitrary.
\begin{itemize}
\item[\textup{(a)}] Given any $\D\in\DD_3$, the subspace $\L_r$ is functionally $b_\D(r)$-vanishing if and only if $\L_r\in\S_3[\D]$.
\item[\textup{(b)}] The subspace $\L_r$ is functionally nonvanishing if and only if $\L_r\in\R_3^+$; whence $\L\in\R_3^+$ follows.
\end{itemize}
\end{pro}

Let us consider first the exceptional cases (E1), (E2), (E3).

\begin{lem}
\label{special-Qr-cases123}
Let us assume that the subspaces $\L_1, \L_2, \L_3$ are not functionally nonvanishing.

\noindent
\textup{(a)} If \textup{(E1)} holds, then
\[Q_1= E_{1,3}+ q_1 E_{1,2}, \; Q_2= E_{2,3} + q_2 E_{2,2}, \; Q_3= E_{3,3} +q_3 E_{3,1}\; (q_1,q_2,q_3\in \C).\]
\textup{(b)} If \textup{(E2)} holds, then
\[Q_1= E_{1,3}+ q_1 E_{1,2}, \; Q_2= E_{2,3} + q_2 E_{2,1}, \; Q_3= E_{3,3} +q_3 E_{3,2}\; (q_1,q_2,q_3\in \C).\]
\textup{(c)} If \textup{(E3)} holds, then
\[Q_1= E_{1,3}+ q_1 E_{1,1}, \; Q_2= E_{2,3} + q_2 E_{2,2}, \; Q_3= E_{3,3} +q_3 E_{3,2}\; (q_1,q_2,q_3\in \C).\]
\end{lem}

\noindent
{\bf Proof.}
Assume that (E1) holds, that is
\[L_1= E_{1,1} + c_1 E_{1,2},\; L_2 = E_{2,1} + c_2 E_{2,2},\; L_3 = E_{3,2}\; \;(c_1,c_2\in\C_0).\]
By Lemma \ref{6dimension-extended-basis}, $Q_1$ is of the form
\[Q_1=\left[\begin{matrix}
0& q_{1,2}^1&1\\
0& q_{2,2}^1& 0\\
q_{3,1}^1&0&0\end{matrix}\right].\]
The general element of $\L_1$ is of the form
\[X= z_1L_1+z_2L_2+z_3L_3+w_1Q_1=
\left[\begin{matrix}
z_1 & z_1c_1+w_1q_{1,2}^1& w_1\\
z_2& z_2c_2 + w_1q_{2,2}^1&0\\
w_1 q_{3,1}^1& z_3&0\end{matrix}\right].\]
For any $b\in\C_0$, we have
\[F_b(X_{P_1})= z_2z_3(1+bc_2)- z_2w_1 c_2 q_{3,1}^1 + z_3w_1 b q_{2,2}^1 - w_1^2 q_{2,2}^1 q_{3,1}^1.\]
Since $\L_1$ is not functionally nonvanishing, this polynomial is identically zero for some $b\in\C_0$.
Hence we infer that $q_{3,1}^1 = q_{2,2}^1=0$.

The general element of $\L_2$ is of the form
\[X=
\left[\begin{matrix}
z_1& c_1z_1 + w_2 q_{1,2}^2& 0\\
z_2& c_2z_2 + w_2 q_{2,2}^2& w_2\\
w_2q_{3,1}^2&z_3&0\end{matrix}\right].\]
For any $b\in\C_0$, we have
\[ F_b(X_{P_2})= z_1z_3(1+bc_1)-z_1w_2 c_1q_{3,1}^2 + z_3w_2 b q_{1,2}^2 - w_2^2 q_{1,2}^2 q_{3,1}^2.\]
Since $\L_2$ is not functionally nonvanishing, we infer that $q_{3,1}^2= q_{1,2}^2=0$.

The general element of $\L_3$ is of the form
\[X=\left[\begin{matrix}
z_1& c_1z_1+w_3 q_{1,2}^3& 0\\
z_2& c_2z_2+w_3q_{2,2}^3&0\\
w_3q_{3,1}^3& z_3&w_3\end{matrix}\right].\]
For any $b\in\C_0$, we have
\begin{eqnarray*} F_b(X_{P_3})&=& z_1z_2(c_2-c_1+bc_1c_2)+ z_1w_3 q_{2,2}^3 (1+bc_1)\\
 & & + z_2w_3 q_{1,2}^3 (-1+bc_2)+ w_3^2b q_{1,2}^3 q_{2,2}^3.
\end{eqnarray*}
Since $\L_3$ is not functionally nonvanishing, all coefficients here must be zero for some $b\in\C_0$.
Applying Lemma \ref{simultaneous-2equations} with the substitutions
\[l_2=c_2,\; l_1=c_1,\; l_3=q_{2,2}^3;\; a_1(b)= a_{1,0,0,1}(b),\; a_2(b)=a_{1,1,0,0}(b)\]
and
\[ l_2=c_2, \; l_1=c_1,\; l_3= q_{1,2}^3;\; \wt a_1(b)= a_{0,1,0,1}(b), \; a_2(b)= a_{1,1,0,0}(b)\]
we obtain $q_{2,2}^3=0$ and $q_{1,2}^3=0$, respectively.

The exceptional cases (E2) and (E3) can be handled similarly. 
(They differ from (E1) only in the order of rows.)

\rightline{$\square$}

\begin{pro}
\label{regularity-cases123}
In the exceptional cases \textup{(E1), (E2), (E3)} we have $\L\in\R_3^+$.
\end{pro}

\noindent
{\bf Proof.}
We may assume that (E1) holds, the other two cases can be settled similarly.

Let us assume that the subspaces $\L_1, \L_2, \L_3$ are not functionally nonvanishing.
We infer by Lemma \ref{special-Qr-cases123} that
\[Q_1= E_{1,3}+ q_1 E_{1,2}, \; Q_2= E_{2,3}+q_2 E_{2,2}, \; Q_3 = E_{3,3} + q_3 E_{3,1}\; (q_1, q_2, q_3\in\C).\]
We recall that transitivity of $\L$ means that for every nonzero $y= \ol\et_1 e_1 + \ol\et_2 e_2 + \ol\et_3 e_3$, the system of equations
\[\la L_kx,y\ra =0 \; \;  (k\in\N_3), \quad \la Q_rx,y\ra=0\;\; (r\in\N_3)\]
has only the trivial zero solution for the vector $x= \x_1e_1 +\x_2 e_2 + \x_3 e_3$.
The special forms of these matrices yield
\[L_1x= (\x_1+c_1\x_2)e_1, \quad L_2x= (\x_1+c_2\x_2)e_2, \quad L_3x= \x_2 e_3,\]
\[Q_1x= (q_1\x_2+\x_3)e_1, \quad Q_2x= (q_2\x_2+\x_3)e_2, \quad Q_3x= (q_3\x_1+\x_3)e_3.\]
Thus, the resulting system of equations takes the form
\[\textup{(1)}\; (\x_1+c_1\x_2)\et_1=0, \quad \textup{(2)}\; (\x_1+c_2\x_2)\et_2=0, \quad \textup{(3)}\; \x_2\et_3=0,\]
\[\textup{(4)}\; (q_1\x_2+\x_3)\et_1=0,\quad \textup{(5)}\; (q_2\x_2+\x_3)\et_2=0, \quad \textup{(6)}\; (q_3\x_1+\x_3)\et_3=0.\]
Choosing $\et_1=\et_2=0$ and $\et_3\ne0$, the equations (1), (2), (4), (5) hold for every $x$.
On the other hand, from (3) and (6) we obtain $\x_2=0$ and $q_3\x_1+\x_3=0$.
Since these have nonzero solutions, the subspace $\L$ is not transitive, what contradicts our standing assumption.
Thus $\L_r$ must be fuctionally nonvanishing for some $r\in\N_3$, and so $\L\in\R_3^+$ is true by Proposition \ref{functional-conditions-6dim}.

\rightline{$\square$}

Now we turn to the exceptional cases (E4) and (E5).

\begin{lem}
\label{special-Qr-cases45}
Assume that the subspaces $\L_1, \L_2, \L_3$ are not functionally nonvanishing.
If \textup{(E4)} or \textup{(E5)} hold, then
\[Q_r =E_{r,3} + q_r E_{r,2}\quad (q_r\in\C, r\in\N_3).\]
\end{lem}

\noindent
{\bf Proof.}
In view of symmetry it is enough to deal with the case $r=1$.

Let us assume first that (E4) holds.
By Lemma \ref{6dimension-extended-basis}, $Q_1$ is of the form
\[Q_1=\left[\begin{matrix}
0& q_{1,2}^1&1\\
0&q_{2,2}^1&0\\
0& q_{3,2}^1&0\end{matrix}\right].\]
The general element of $\L_1$ is of the form
\[X=z_1L_1+z_2L_2+z_3L_3+w_1Q_1=
\left[\begin{matrix}
z_1 & w_1q_{1,2}^1& w_1\\
z_2& w_1 q_{2,2}^1&0\\
z_3& w_1q_{3,2}^1&0\end{matrix}\right].\]
For any $b\in\C_0$, we have
\[F_b(X_{P_1})= z_2 w_1 q_{3,2}^1 - z_3w_1 q_{2,2}^1 + w_1^2 b q_{2,2}^1 q_{3,2}^1.\]
Since $\L_1$ is not functionally nonvanishing, we infer that $q_{3,2}^1= q_{2,2}^1=0$.

Let us assume now that (E5) holds.
The general element of $\L_1$ is of the form
\[X= \left[\begin{matrix}
z_1& z_1c_1 + w_1 q_{1,2}^1& w_1\\
z_2& z_2c_2+ w_1q_{2,2}^1& 0\\
z_3& z_3c_3+ w_1 q_{3,2}^1&0\end{matrix}\right].\]
For any $b\in\C_0$, we have
\begin{eqnarray*}
F_b(X_{P_1})&=& z_2z_3 (c_3-c_2 +b c_2c_3) + z_2 w_1 q_{3,2}^1 (1+ b c_2)\\
 & & + z_3 w_1 q_{2,2}^1(-1+ bc_3) + w_1^2 b q_{2,2}^1 q_{3,2}^1.
\end{eqnarray*}Tking into accounts that $c_2, c_3$ are distinc nonzero numbers, we infer by Lemma \ref{simultaneous-2equations} that $q_{3,2}^1= q_{2,2}^1=0$.

\rightline{$\square$}

\begin{pro}
\label{regularity-cases45}
In the exceptional cases \textup{(E4)} and \textup{(E5)} we have $\L\in\R_3^+$.
\end{pro}

\noindent
{\bf Proof.}
Let us assume that the subspaces $\L_1, \L_2, \L_3$ are not functionally nonvanishing.
We infer by Lemma \ref{special-Qr-cases45} that $Q_r= E_{r,3}+ q_r E_{r,2} \; (q_r\in\C)$ holds, for every $r\in\N_3$.
Now we check the system of equations testing transitivity of $\L$.

If (E4) holds then this system takes the form
\[\textup{(1)}\; \x_1\et_1=0, \quad \textup{(2)}\; \x_1\et_2=0, \quad \textup{(3)} \; \x_1\et_3=0,\]
\[ \textup{(4)}\; (q_1\x_2+\x_3)\et_1=0, \quad \textup{(5)}\; (q_2\x_2+\x_3)\et_2=0, \quad \textup{(6)} \; (q_3\x_2+\x_3)\et_3=0.\]Choosing $\et_1=\et_2=0$ and $\et_3\ne 0$, this system reduces to
\[\x_1=0, \quad q_3 \x_2+\x_3=0,\]
which has nonzero solutions for $x$.
Hence $\L$ is not transitive.

If (E5) holds then the corresponding system is
\[\textup{(1)}\; (\x_1+c_1\x_2)\et_1=0, \quad \textup{(2)}\; (\x_1+ c_2\x_2)\et_2=0, \quad \textup{(3)}\; (\x_1+c_3\x_2)\et_3=0,\]
\[\textup{(4)}\; (q_1\x_2+\x_3)\et_1=0, \quad \textup{(5)}\; (q_2\x_2+\x_3)\et_2=0, \quad \textup{(6)}\; (q_3\x_2+\x_3)\et_3=0.\]
Choosing $\et_1=\et_2=0$ and $\et_3\ne0$, this system reduces to
\[\x_1+c_3\x_2=0, \quad q_3\x_2 + \x_3=0,\]
which has nonzero solutions for $x$.

Therefore, in both cases we conclude that $\L$ is not transitive, what contradicts our standing assumption.
Consequently, $\L_r$ must be functionally nonvanishing for some $r\in\N_3$, and so $\L\in\R_3^+$ is true by Proposition \ref{functional-conditions-6dim}.

\rightline{$\square$}

We have arrived at the main result of this section.

\begin{thm}
\label{regularity-6dimension}
If $\L\in\T_3$ and $\dim\L=6$, then $\L\in\R_3^+$ is true.
\end{thm}

\noindent
{\bf Proof.}
Summing up the previous results, recall that if $\L\in\T_3$ and $\L_0$ is functionally nonvanishing, then $\L\in\R_3^+$ holds by Corollary \ref{strong-regularity-condition}.
Propositions \ref{classI-3basis}, \ref{classII-3basis}, \ref{classIII-3basis}, \ref{classIV-3basis} tell us that $\L_0$ is functionally nonvanishing with the exceptions (E1)--(E5).
By Proposition \ref{functional-conditions-6dim}, if $\L_r$ is functionally nonvanishing for some $r\in\N_3$, then $\L\in\R_3^+$.
Relying on this statement it is shown in Propositions \ref{regularity-cases123}, \ref{regularity-cases45} that $\L\in\R_3^+$ holds also in the exceptional cases (E1)--(E5).

\rightline{$\square$}

\section{5-dimensional transitive subspaces}
\label{5dim-transitive}

\emph{Let us be given $\L\in\T_3$ with $\dim\L=5$.}
We are going to follow the strategy, which was applied in 6-dimension.
We omit details if they are trivial modifications of the 6-dimensional case.

The 2-dimensional subspace $\L_0=\L\cap\M_3$ plays significant role in this situation too.
Recall from Section \ref{subspaces-of-matrices} (see Proposition \ref{regularity-condition}):

$\L_0$ is functionally $\D$-vanishing: $\fa r\in\N_3,\; \fa L\in\L_0, \; F_{b_\D(r)}(L_{P_r})=0$

$\quad (\D\in\DD_3)$;

$\L_0$ is functionally nonvanishing: $\exists r\in\N_3,\; \fa b\in\C_0, \; \exists L\in\L_0, \; F_b(L_{P_r})\ne0$;

$\L_0$ is not functionally $\D$-vanishing $\;\imp\; \L\in \R_3[\D]$;

$\L_0$ is functionally nonvanishing $\;\imp$

 $\quad\L_0$ is not functionally $\D$-vanishing $\fa\D\in\DD_3 \;\imp \; \L\in\R_3^+.$

\noindent
Let $\J_2$ stand for the set of increasing pairs of elements in $\N_3\times\N_2$.

\begin{lem}
\label{canonical-basis-2dim}
There exist unique $J=((i_1,j_1),(i_2,j_2))\in\J_2$ and matrices $L_k= [l_{i,j}^k]_3\in \L_0\; (k\in\N_2)$ such that
\[l_{i_r,j_r}^k=\d_{k,r}\quad (k,r\in\N_2)\]
and
\[l_{i,j}^k=0\quad\hbox{if}\; (i,j)\prec (i_k,j_k) \; (k\in\N_2).\]
\end{lem}

The pair $(L_1,L_2)$ and the subspace $\L_0$ itself are called  of \emph{type} $J$.
The general element of $\L_0$ is of the form
\[L=z_1L_1 + z_2L_2 \quad (z_1,z_2\in\C).\]
For any $r\in\N_3$ and $b\in\C_0$, we have
\[F_b(L_{P_r})= p_{r,b}(z_1,z_2)= \sum a_{k_1,k_2}(r,b) z_1^{k_1} z_2^{k_2}.\]
This polynomial is identically zero if and only if the coefficients $a_{k_1,k_2}(r,b)$ are zero for all $(k_1,k_2)\in \Z_+^2$.

In order to carry out a systematic study we classify the types in $\J_2$, each class containing 3 pairs:
\[\textup{(I)} \; j_1=j_2=2, \quad \textup{(II)} \; j_1=1, j_2=2, i_1\ne i_2, \quad \textup{(III)} \; j_1=2, j_2=1,\]
\[\textup{(IV)} \; i_1=i_2, \quad \textup{(V)} \; j_1=j_2=1.\]
We are going to examine validity of the previous functionally nonvanishing properties in the case of each type.

\begin{pro}
\label{classI-2basis}
If the type of $\L_0$ is of class \textup{(I)}, then $\L_0$ is functionally nonvanishing, and so $\L\in\R_3^+$.
\end{pro}

\noindent
{\bf Proof.}
Suppose that $\L_0$ is of type $((1,2),(2,2))$.
Then
\[L=z_1L_1+z_2L_2 = \left[
\begin{matrix}
0& z_1&0\\
z_1 l_{2,1}^1& z_2&0\\
z_1 l_{3,1}^1 + z_2 l_{3,1}^2 & z_1 l_{3,2}^1 +z_2 l_{3,2}^2&0\end{matrix}\right],\]
and so
\[F_b(L_{P_3})= -l_{2,1}^1 z_1^2 + b z_1 z_2,  \quad \hbox{with}\; a_{1,1}=b\ne0.\]
If the type is $((1,2),(3,2))$, then $F_b(L_{P_2})= -z_1^2 l_{3,1}^1 + b z_1z_2$.
Finally, if the type is $((2,2),(3,2))$, then $F_b(L_{P_1})= -z_1^2 l_{3,1}^1 + b z_1z_2$.

\rightline{$\square$}

Before turning to the other classes, we note that for every $\D = \hbox{diag}(b_1,b_2,b_3)$ in $\DD_3$, we have 
\[b_\D(1) + b_\D(3)= b_\D(2).\]

\begin{pro}
\label{classII-2basis}
Let $\D\in\DD_3$ be arbitrary.
Suppose that the type $J$ of $\L_0$ is of class \textup{(II)}.
Then $\L_0$ is functionally $\D$-vanishing if and only if

\medskip
\textup{(a)} $J=((1,1),(2,2))$:

\medskip
\noindent
\textup{(E1)} $\quad L_1= \left[\begin{matrix}
1& -1/b_\D(3) & 0\\
0&0&0\\
b_\D(1)x & x&0\end{matrix}\right], \quad 
L_2= \left[\begin{matrix}
0&0&0\\
0&1&0\\
b_\D(1)y&y&0\end{matrix}\right] \quad (x,y\in\C);$

\medskip
\textup{(b)} $J=((1,1), (3,2))$:

\medskip
\noindent
\textup{(E2)} $\quad L_1=
\left[\begin{matrix}
1& -1/b_\D(2)& 0\\
-b_\D(1)x & x&0\\
0&0&0\end{matrix}\right], \quad
L_2= \left[\begin{matrix}
0&0&0\\
0&0&0\\
0&1&0\end{matrix}\right] \quad (x\in\C);$

\medskip
\textup{(c)} $J=((2,1),(3,2))$:

\medskip
\noindent
\textup{(E3)} $\quad L_1=
\left[\begin{matrix}
0&0&0\\
1& -1/b_\D(1)& 0\\
0&0&0\end{matrix}\right], \quad L_2=
\left[\begin{matrix}
0&0&0\\
0&0&0\\
0&1&0\end{matrix}\right].$
\end{pro}

\noindent
{\bf Proof.}
(a): By Lemma \ref{canonical-basis-2dim}, we know that
\[L_1=
\left[\begin{matrix}
1& l_{1,2}^1 &0\\
l_{2,1}^1 &0&0\\
l_{3,1}^1 & l_{3,2}^1&0\end{matrix}\right] \quad \hbox{and} \quad L_2=
\left[\begin{matrix}
0&0&0\\
0&1&0\\
l_{3,1}^2& l_{3,2}^2&0\end{matrix}\right],\]
hence
\[ L= z_1L_1+z_2L_2 = \left[\begin{matrix}
z_1& z_1 l_{1,2}^1&0\\
z_1 l_{2,1}^1& z_2&0\\
z_1 l_{3,1}^1 +z_2 l_{3,1}^2& z_1 l_{3,2}^1 + z_2 l_{3,2}^2&0\end{matrix}\right].\]
Suppose that $\L_0$ is functionally $\D$-vanishing.
The polynomial 
\[F_{b_\D(3)}(L_{P_3}) = -z_1^2 l_{1,2}^1 l_{2,1}^1 + z_1z_2(1+ b_\D(3) l_{1,2}^1)\]
is identically zero exactly when $l_{1,2}^1= -1/b_\D(3)$ and $l_{2,1}^1=0$.
Inview of $l_{2,1}^1=0$, we obtain
\[F_{b_\D(1)}(L_{P_1})= z_1z_2 (-l_{3,1}^1 + b_\D(1) l_{3,2}^1) + z_2^2 (-l_{3,1}^2 + b_\D(1) l_{3,2}^2),\]
which is identically zero if and only if $l_{3,1}^1= b_\D(1) l_{3,2}^1$ and $l_{3,1}^2 = b_\D(1) l_{3,2}^2$.
Taking $l_{3,2}^1=x$ and $l_{3,2}^2=y$, we arrive at the form (E1).

It remains to check that if $(L_1,L_2)$ is of the form (E1), then the third polynomial is automatically zero.
Indeed, then
\[L= \left[
\begin{matrix}
z_1& -z_1/b_\D(3) &0\\
0& z_2&0\\
z_1 b_\D(1)x +z_2b_\D(1)y & z_1x+z_2y & 0\end{matrix}\right],\]
and so
\begin{eqnarray*}
F_{b_\D(2)}(L_{P_2}) & = & z_1(z_1x+z_2y) + \frac{z_1}{b_\D(3)} \left(z_1 b_\D(1)x + z_2 b_\D(1)y\right)\\
 & & - \frac{b_\D(2)}{b_\D(3)} z_1(z_1x+z_2y)\\
 & = & z_1^2\frac{x}{b_\D(3)} \left( b_\D(3)+ b_\D(1) - b_\D(2)\right)\\
 & & + z_1z_2 \frac{y}{b_\D(3)} \left(b_\D(3)+ b_\D(1)-b_\D(2)\right) \equiv 0.
\end{eqnarray*}

The proofs of the other two cases (b) and (c) are similar and are left to the reader.

\rightline{$\square$}

\begin{pro}
\label{classIII-2basis}
Let $\D\in\DD_3$ be arbitrary.
Suppose that the type $J$ of $\L_0$ is of class \textup{(III)}.
Then $\L_0$ is functionally $\D$-vanishing if and only if

\medskip
\textup{(a)} $J=((1,2), (2,1))$:

\medskip
\noindent
\textup{(E4)} $\quad L_1=
\left[\begin{matrix}
0&1&0\\
0&0&0\\
b_\D(2)x& x&0\end{matrix}\right], \quad
L_2= \left[\begin{matrix}
0&0&0\\
1&1/b_\D(3)& 0\\
b_\D(2)y &y&0\end{matrix}\right] \quad (x,y\in\C)$;

\medskip
\textup{(b)} $J=((1,2),(3,1))$:

\medskip
\noindent
\textup{(E5)} $\quad L_1=
\left[\begin{matrix}
0&1&0\\
b_\D(3)x&x&0\\
0&0&0\end{matrix}\right], \quad L_2=
\left[\begin{matrix}
0&0&0\\
0&0&0\\
1& 1/b_\D(2)&0\end{matrix}\right] \quad (x\in\C)$;

\medskip
\textup{(c)} $J= ((2,2), (3,1))$:

\medskip
\noindent
\textup{(E6)} $\quad L_1=
\left[\begin{matrix}
0&0&0\\
0&1&0\\
0&0&0\end{matrix}\right], \quad
L_2= \left[\begin{matrix}
0&0&0\\
0&0&0\\
1& 1/b_\D(1)& 0\end{matrix}\right]$.
\end{pro}

\noindent
{\bf Proof.}
(a): By Lemma \ref{canonical-basis-2dim} we know
\[ L_1= \left[\begin{matrix}
0&1&0\\
0&l_{2,2}^1&0\\
l_{3,1}^1& l_{3,2}^1&0\end{matrix}\right], \quad L_2 = \left[\begin{matrix}
0&0&0\\
1&l_{2,2}^2&0\\
l_{3,1}^2&l_{3,2}^2&0\end{matrix}\right],\]
hence
\[L=
\left[\begin{matrix}
0&z_1&0\\
z_2& z_1l_{2,2}^1 + z_2l_{2,2}^2&0\\
z_1l_{3,1}^1 + z_2 l_{3,1}^2 & z_1 l_{3,2}^1 + z_2l_{3,2}^2&0\end{matrix}\right].\]
Suppose that $\L_0$ is functionally $\D$-vanishing.
The polynomial 
\[F_{b_\D(3)}(L_{P_3})= z_1^2 b_\D(3) l_{2,2}^1 + z_1z_2 (-1 + b_\D(3) l_{2,2}^2)\]
is identically zero exactly when $l_{2,2}^1=0$ and $l_{2,2}^2 = 1/b_\D(3)$.
Furthermore, the polynomial 
\[F_{b_\D(2)}(L_{P_2}) = z_1^2(-l_{3,1}^1 + b_\D(2) l_{3,2}^1) + z_1z_2 (-l_{3,1}^2 + b_\D(2)l_{3,2}^2)\]
is identically zero if and only if $l_{3,1}^1 = b_\D(2) l_{3,2}^1$ and $l_{3,1}^2 = b_\D(2) l_{3,2}^2$.
Taking $l_{3,2}^1=x$ and $l_{3,2}^2=y$, we arrive at the form (E4).

Finally, if (E4) holds, then
\begin{eqnarray*}
F_{b_\D(1)}(L_{P_1})&=& z_1z_2 \frac{x}{b_\D(3)} (b_\D(3)-b_\D(2)
  + b_\D(1)) \\
  & &+ z_2^2 \frac{y}{b_\D(3)} (b_\D(3)- b_\D(2)+ b_\D(1)) \equiv 0.\end{eqnarray*}

The proof of the statements in (b) and (c) are similar and are left to the reader.

\rightline{$\square$}

\begin{pro}
\label{classIV-2basis}
Let $\D\in\DD_3$ be arbitrary.
Suppose that the type $J$ of $\L_0$ is of class \textup{(IV)}.
Then $\L_0$ is functionally $\D$-vanishing if and only if

\medskip
\textup{(a)} $J=((1,1),(1,2))$:

\medskip
\noindent
\textup{(E7)} $\quad L_1=
\left[\begin{matrix}
1&0&0\\
x&0&0\\
y&0&0\end{matrix}\right], \quad 
L_2= \left[\begin{matrix}
0&1&0\\
b_\D(3)x & x&0\\
b_\D(2)y & y&0\end{matrix}\right] \quad (x,y\in\C);$

\medskip
\textup{(b)} $J=((2,1),(2,2))$:

\medskip
\noindent
\textup{(E8)} $\quad L_1=
\left[\begin{matrix}
0&0&0\\
1&0&0\\
x&0&0\end{matrix}\right], \quad L_2=
\left[\begin{matrix}
0&0&0\\
0&1&0\\
b_\D(1)x&x&0\end{matrix}\right]\quad (x\in\C)$;

\medskip
\textup{(c)} $J=((3,1),(3,2))$:

\medskip
\noindent
\textup{(E9)} $\quad L_1=
\left[\begin{matrix}
0&0&0\\
0&0&0\\
1&0&0\end{matrix}\right], \quad L_2=
\left[\begin{matrix}
0&0&0\\
0&0&0\\
0&1&0\end{matrix}\right].$
\end{pro}

\noindent
{\bf Proof.}
(a): By Lemma \ref{canonical-basis-2dim} we know
\[L_1= \left[\begin{matrix}
1&0&0\\
l_{2,1}^1&l_{2,2}^1&0\\
l_{3,1}^1&l_{3,2}^1&0\end{matrix}\right], \quad
L_2= \left[\begin{matrix}
0&1&0\\
l_{2,1}^2& l_{2,2}^2&0\\
l_{3,1}^2&l_{3,2}^2&0\end{matrix}\right],\]
hence
\[L=\left[\begin{matrix}
z_1& z_2&0\\
z_1 l_{2,1}^1 + z_2 l_{2,1}^2 & z_1 l_{2,2}^1 + z_2 l_{2,2}^2&0\\
z_1 l_{3,1}^1 + z_2 l_{3,1}^2 & z_1 l_{3,2}^1+ z_2 l_{3,2}^2&0\end{matrix}\right].\]Suppose that $\L_0$ is functionally $\D$-vanishing.
The polynomial 
\begin{eqnarray*}
F_{b_\D(3)}(L_{P_3})&=& z_1^2 l_{2,2}^1 + z_1z_2 (l_{2,2}^2 - l_{2,1}^1 + b_\D(3) l_{2,2}^1)\\
 & & + z_2^2 (-l_{2,1}^2 + b_\D(3) l_{2,2}^2)
\end{eqnarray*}
is identically zero if and only if
\[l_{2,2}^1=0, \quad l_{2,1}^1= l_{2,2}^2, \quad l_{2,1}^2 = b_\D(3)l_{2,2}^2.\]
Furthermore,
\begin{eqnarray*}
F_{b_\D(2)}(L_{P_2}) & = & z_1^2 l_{3,2}^1 + z_1z_2 (l_{3,2}^2 - l_{3,1}^1 + b_\D(2) l_{3,2}^1) \\
 & & + z_2^2 (-l_{3,1}^2 + b_\D(2) l_{3,2}^2)\end{eqnarray*}
is identically zero exactly when
\[l_{3,2}^1 = 0, \quad l_{3,2}^2 = l_{3,1}^1, \quad l_{3,1}^2 = b_\D(2) l_{3,2}^2.\]
Taking $l_{2,2}^2 =x$ and $l_{3,2}^2= y$, we obtain the form (E7).

It is easy to check that if (E7) holds, then $F_{b_\D(1)}(L_{P_1})\equiv 0$ is also true.

The proof of (b) and (c) is left to the reader.

 \rightline{$\square$}

\begin{pro}
\label{classV-2basis}
Let $\D\in\DD_3$ be arbitrary.
Suppose that the type $J$ of $\L_0$ is of class \textup{(V)}.
Then $\L_0$ is functionally $\D$-vanishing if and only if

\medskip
\textup{(a)} $J=((1,1), (2,1))$:

\medskip
\noindent
\textup{(E10)} $\quad L_1=
\left[\begin{matrix}
1&0&0\\
0&0&0\\
x&0&0\end{matrix}\right], \quad
L_2= \left[\begin{matrix}
0&0&0\\
1&0&0\\
y&0&0\end{matrix}\right] \quad (x, y\in\C)$,

\medskip
or

\medskip
\noindent
\textup{(E11)} $\quad L_1=
\left[\begin{matrix}
1&u&0\\
0&0&0\\
(1+b_\D(2)u)x/u& x&0\end{matrix}\right], \quad (u\in\C\setminus\{0, -1/b_\D(3)\}, x,y\in\C)$

\medskip

\hskip2cm $L_2=
\left[\begin{matrix}
0&0&0\\
1& u/(1+b_\D(3)u)& 0\\
(1+b_\D(2)u)y/u& y &0\end{matrix}\right]$;

\medskip
\textup{(b)} $J=((1,1), (3,1)):$

\medskip
\noindent
\textup{(E12)} $\quad L_1=
\left[\begin{matrix}
1&0&0\\
x&0&0\\
0&0&0\end{matrix}\right], \quad L_2=
\left[\begin{matrix}
0&0&0\\
0&0&0\\
1&0&0\end{matrix}\right]\quad (x\in\C),$

\medskip

or 

\medskip  

\noindent
\textup{(E13)}
$\quad L_1=\left[\begin{matrix}
1&u&0\\
(1+b_\D(3)u)x/u & x&0\\
0&0&0\end{matrix}\right],$

\medskip

\hskip3cm $L_2=
\left[\begin{matrix}
0&0&0\\
0&0&0\\
1& u/(1+b_\D(2)u)&0\end{matrix}\right]$

\medskip
\hskip4cm $ (u\in \C_0\setminus\{ -1/b_\D(3), -1/b_\D(2)\},\; x\in\C); $

\medskip
\textup{(c)} $J=((2,1),(3,1))$:

\medskip
\noindent
\textup{(E14)} $\quad L_1= 
\left[\begin{matrix}
0&0&0\\
1&x&0\\
0&0&0\end{matrix}\right], \quad L_2=
\left[\begin{matrix}
0&0&0\\
0&0&0\\
1& x/(1+b_\D(1)x)& 0\end{matrix}\right]$

\medskip
\hskip6cm $(x\in\C\setminus\{-1/b_\D(1)\}).$

\end{pro}

\noindent
{\bf Proof.}
(a): By Lemma \ref{canonical-basis-2dim} we know
\[L_1= \left[\begin{matrix}
1& l_{1,2}^1&0\\
0& l_{2,2}^1&0\\
l_{3,1}^1&l_{3,2}^1&0\end{matrix}\right], \quad L_2=
\left[\begin{matrix}
0&0&0\\
1&l_{2,2}^2&0\\
l_{3,1}^2&l_{3,2}^2&0\end{matrix}\right],\]
hence
\[L=\left[\begin{matrix}
z_1& z_1l_{1,2}^1&0\\
z_2& z_1l_{2,2}^1 + z_2 l_{2,2}^2&0\\
z_1 l_{3,1}^1 + z_2 l_{3,1}^2 & z_1 l_{3,2}^1 + z_2l_{3,2}^2&0\end{matrix}\right].\]
Suppose that $\L_0$ is $\D$-vanishing.
The polynomial 
\[F_{b_\D(3)}(L_{P_3})= z_1^2 l_{2,2}^1 (1+ b_\D(3) l_{1,2}^1) + z_1z_2 (l_{2,2}^2 - l_{1,2}^1 +b_\D(3) l_{1,2}^1 l_{2,2}^2)\]
is identically zero if and only if
\[l_{2,2}^1 (1+ b_\D(3) l_{1,2}^1)=0 \quad \hbox{and} \quad l_{1,2}^1 = l_{2,2}^2 (1+ b_\D(3) l_{1,2}^1).\]
The assumption $1+b_\D(3)l_{1,2}^1=0$ leads to the contradiction $l_{1,2}^1=0$.
Hence $1+ b_\D(3) l_{1,2}^1\ne 0$, and so
\[l_{2,2}^1=0 \quad \hbox{and} \quad l_{2,2}^2 = \frac{l_{1,2}^1}{1+ b_\D(3)l_{1,2}^1}.\]
The second polynomial 
\begin{eqnarray*}
F_{b_\D(2)}(L_{P_2}) & = & z_1^2 (l_{3,2}^1 - l_{1,2}^1 l_{3,1}^1 + b_\D(2) l_{1,2}^1 l_{3,2}^1)\\
 & & +z_1z_2 (l_{3,2}^2 - l_{1,2}^1 l_{3,1}^2 + b_\D(2) l_{1,2}^1 l_{3,2}^2)\end{eqnarray*}
is identically zero exactly when
\[l_{1,2}^1 l_{3,1}^1= l_{3,2}^1(1+ b_\D(2) l_{1,2}^1), \quad l_{1,2}^1 l_{3,1}^2 = l_{3,2}^2 (1+ b_\D(2) l_{1,2}^1).\]
Assuming $l_{1,2}^1=0$, we infer $l_{2,2}^2 = l_{3,2}^1 = l_{3,2}^2 =0$.
Taking $l_{3,1}^1=x$ and $l_{3,1}^2 =y$, we obtain (E10).
Assuming $l_{1,2}^1\ne 0$, and taking $l_{1,2}^1 =u, l_{3,2}^1=x, l_{3,2}^2 =y$ we obtain (E11).

If (E10) holds, then $F_{b_\D(1)}(L_{P_1})\equiv 0$ because of a zero column.
If (E11) holds, then a short computation yields that $F_{b_\D(1)}(L_{P_1}) \equiv 0$.

The proof of (b) and (c) is left to the reader.

\rightline{$\square$}

\section{The exceptional forms in 5-dimension}
\label{exceptional-forms}

Now we turn to the study of the exceptional forms (E1)--(E14).
We are going to show that $\D$-singularity and transitivity are inconsistent conditions in these cases, with the only exception (E7).

\emph{Let us be given a subspace $\L\in \hbox{{\rm Lat}}M_3[\C]$ with $\dim\L=5$, and let us assume that $\dim\L_0=2$ is true for the subspace $\L_0= \L\cap\M_3$.}
We complete the canonical basis of $\L_0$ to a canonical basis of $\L$.

\begin{lem}
\label{basis-completion-5dim}
Let $(L_1,L_2)$ be a canonical basis in $\L_0$ of type 
\[J=((i_1,j_1), (i_2,j_2))\in\J_2.\]
Then there exist unique matrices $Q_k= [q_{i,j}^k]_3\in \L\;\; (k\in\N_3)$ such that
\[q_{i,3}^k = \d_{i,k}\quad (i,k\in\N_3)\]
and
\[q_{i_r,j_r}^k = 0 \quad (r\in\N_2,\; k\in\N_3).\]
\end{lem}

The \emph{canonical basis} $(L_1,L_2,Q_1,Q_2,Q_3)$ of $\L$, and $\L$ itself, are also called \emph{of type} $J$.

We enlarge the subspace $\L_0$ by 1-dimension.
Namely, for every $r\in\N_3$, let $\L_r$ denote the subspace spanned by $\L_0$ and $Q_r$.
Given any $\D= \hbox{diag}(b_1,b_2,b_3)\in\DD_3$, it is clear that $\D$-singularity of $\L$ implies $\D$-singularity of $\L_r$, for every $r\in\N_3$.
Furthermore, $\L_r\in\S_3[\D]$ holds if and only if $\L_r$ is \emph{functionally $\D$-vanishing}, that is $F_{b_\D(r)}(X_{P_r})=0$ is true, for every $X\in\L_r$.

\begin{thm}
\label{En-inconsistancy}
Let $\D\in\DD_3$ be given.
Let us assume that the canonical basis $(L_1,L_2)$ of $\L_0$ is of the form \textup{(En)}, where $n\in\N_{14}$ but $n\ne 7$.
If $\L$ is $\D$-singular, then $\L$ cannot be transitive: $\L\in\S_3[\D] \imp \L\not\in\T_3$.
\end{thm}

\noindent
{\bf Proof.}
We provide the proof for the form (E1).
The other cases can be treated similarly.
The details are left to the interested reader.

Let us assume that (E1) holds:
\[L_1=\left[\begin{matrix}
1& -1/b_\D(3)&0\\
0&0&0\\
b_\D(1)x& x&0\end{matrix}\right], \quad L_2=
\left[\begin{matrix}
0&0&0\\
0&1&0\\
b_\D(1)y&y&0\end{matrix}\right] \quad (x,y\in\C).\]
It follows by Lemma \ref{basis-completion-5dim} that
\[ Q_r=
\left[\begin{matrix}
0& q_{1,2}^r& \d_{1,r}\\
q_{2,1}^r& 0& \d_{2,r}\\
q_{3,1}^r& q_{3,2}^r& \d_{3,r}\end{matrix}\right] \quad (r\in\N_3).\]
The general element of $\L_r$ is of the form
\[X=\left[\begin{matrix}
z_1&  -z_1/b_\D(3) + w_r q_{1,2}^r& w_r\d_{1,r}\\
w_r q_{2,1}^r & z_2& w_r\d_{2,r}\\
z_1 b_\D(1)x + z_2 b_\D(1)y + w_r q_{3,1}^r & z_1x +z_2y + w_r q_{3,2}^r&w_r \d_{3,r}\end{matrix}\right].\]
The polynomial
\[F_{b_\D(r)}(X_{P_r})= p_r(z_1,z_2,w_r)\]
is identically zero exactly when all coefficients are zero in its canonical form.

For $r=1$, we have
\[F_{b_\D(1)}(X_{P_1})= z_1w_1 q_{2,1}^1 x + z_2w_1(q_{2,1}^1y - q_{3,1}^1 + b_\D(1) q_{3,2}^1) + w_1^2 q_{2,1}^1 q_{3,2}^1,\]
which is identically zero if and only if
\[q_{2,1}^1x=0, \quad q_{3,1}^1 = q_{2,1}^1 y + b_\D(1) q_{3,2}^1, \quad q_{2,1}^1 q_{3,2}^1 = 0.\]
Assuming $q_{2,1}^1=0$, we get $q_{3,1}^1= b_\D(1) q_{3,2}^1$, and so with $q_{3,2}^1=q_1\in\C$ we obtain
\[Q_1 = \left[\begin{matrix}
0& q_{1,2}^1& 1\\
0&0&0\\
b_\D(1)q_1& q_1&0\end{matrix}\right].\]
Assuming $q_{2,1}^1\ne 0$, we get $x=q_{3,2}^1=0,\; q_{3,1}^1 = q_{2,1}^1 y$, and so with $q_{2,1}^1=q_1$ we obtain
\[Q_1'=\left[\begin{matrix}
0& q_{1,2}^1 & 1\\
q_1&0&0\\
q_1y& 0&0\end{matrix}\right] \quad (q_1\in\C_0,\; x=0).\]
For $r=2$, we have
\begin{eqnarray*}
F_{b_\D(2)}(X_{P_2}) & = & z_1w_2 \left(q_{3,2}^2 + \frac{1}{b_\D(3)} q_{3,1}^2 - b_\D(1)x q_{1,2}^2 - \frac{b_\D(2)}{b_\D(3)} q_{3,2}^2 + b_\D(2) q_{1,2}^2 x\right) \\
 & & + z_2 w_2 \left( -q_{1,2}^2 b_\D(1)y + b_\D(2) q_{1,2}^2 y\right)\\
 & & + w_2^2 \left(-q_{1,2}^2 q_{3,1}^2 + b_\D(2) q_{1,2}^2 q_{3,2}^2\right),
\end{eqnarray*}
which is identically zero exactly when
\[b_\D(3) q_{1,2}^2x -\frac{b_\D(1)}{b_\D(3)} q_{3,2}^2 + \frac{1}{b_\D(3)} q_{3,1}^2=0,\]
\[q_{1,2}^2y=0, \quad q_{1,2}^2(-q_{3,1}^2 + b_\D(2) q_{3,2}^2)=0.\]
We have applied the identity $b_\D(3)+ b_\D(1)= b_\D(2)$.
Assuming $q_{1,2}^2=0$, we get $b_\D(1) q_{3,2}^2 = q_{3,1}^2$, and so with $q_{3,2}^2 =q_2$ we obtain
\[Q_2=\left[\begin{matrix}
0&0&0\\
q_{2,1}^2&0&1\\
b_\D(1) q_2 & q_2&0\end{matrix}\right].\]
Assuming $q_{1,2}^2\ne0$, we get $y=0, \; q_{3,1}^2= b_\D(2) q_{3,2}^2, \; b_\D(3) q_{1,2}^2x + q_{3,2}^2=0$, and so with $q_{1,2}^2=q_2$ we obtain
\[Q_2'= 
\left[\begin{matrix}
0& q_2&0\\
q_{2,1}^2 &0&1\\
-b_\D(2)b_\D(3) q_2x & - b_\D(3) q_2x& 0 \end{matrix}\right] \quad (q_2\in\C_0, \; y=0).\]
For $r=3$, we have
\[F_{b_\D(3)}(X_{P_3}) = z_1 w_3 \frac{1}{b_\D(3)} q_{2,1}^3 + z_2 w_3 b_\D(3) q_{1,2}^3 - w_3^2 q_{1,2}^3 q_{2,1}^3,\]
which is identically zero if and only if $q_{2,1}^3= q_{1,2}^3=0$, hence 
\[Q_3= \left[\begin{matrix}
0&0&0\\
0&0&0\\
q_{3,1}^3&q_{3,2}^3&1\end{matrix}\right].\]

Considering $\wt x= \x_1e_1 + \x_2 e_2 + \x_3e_3$ and $\wt y= \wt\et_1 e_1 + \wt\et_2e_2 + \wt\et_3 e_3\ne0$, the system of equations checking transitivity takes the form
\begin{itemize}
\item[\textup{(1)}] $\quad 0= \la L_1\wt x,\wt y\ra = \left(\x_1-\frac{1}{b_\D(3)}\x_2\right) \et_1 + x (b_\D(1)\x_1+\x_2)\et_3,$
\item[\textup{(2)}] $\quad 0= \la L_2 \wt x, \wt y\ra =\x_2 \et_2 + y(b_\D(1)\x_1 + \x_2)\et_3,$
\item[{(3)}] $\quad 0= \la Q_1 \wt x, \wt y \ra = (q_{1,2}^1 \x_2+\x_3)\et_1 + q_1 (b_\D(1)\x_1 + \x_2)\et_3,$
\item[\textup{(3')}] $\quad 0= \la Q_1' \wt x, \wt y\ra = (q_{1,2}^1\x_2+\x_3)\et_1 + q_1\x_1\et_2+ q_1y\x_1\et_3; \; x=0,$
\item[\textup{(4)}] $\quad 0= \la Q_2 \wt x, \wt y\ra= (q_{2,1}^2 \x_1 +\x_3)\et_2 + q_2 (b_\D(1)\x_1+\x_2)\et_3,$
\item[\textup{(4')}] $\quad 0= \la Q_2' \wt x,\wt y\ra = q_2 \x_2\et_1 + (q_{2,1}^2\x_1+\x_3)\et_2 - b_\D(3) q_2 x (b_\D(2)\x_1+\x_2)\et_3 ; \; y=0,$
\item[\textup{(5)}] $\quad 0 = \la Q_3 \wt x, \wt y\ra = (q_{3,1}^3 \x_1 + q_{3,2}^3 \x_2 + \x_3) \et_3.$
\end{itemize}

If (3) and (4) hold, then choosing $\et_1=\et_2=0$ and $\et_3\ne 0$, the system reduces to
\[b_\D(1)\x_1+\x_2=0,\quad q_{3,1}^3 \x_1 + q_{3,2}^3 \x_2 +\x_3 =0.\]
If (3) and (4') hold with $y=0$, then  choosing $\et_1=\et_3=0$ and $\et_2\ne0$, the system reduces to
\[\x_2=0, \quad q_{2,1}^2 \x_1 +\x_3=0.\]
If (3') with $x=0$ and (4) hold, then choosing $\et_2=\et_3=0$ and $\et_1\ne0$, the system reduces 
to
\[b_\D(3)\x_1 - \x_2=0, \quad q_{1,2}^1 \x_2 +\x_3=0.\]
If (3') and (4') hold with $x=y=0$, then choosing $\et_1=\et_2=0$ and $\et_3\ne0$, the system reduces to
\[q_{3,1}^3\x_1 + q_{3,2}^3 \x_2 + \x_3 =0.\]
All the reduced systems have nonzero solutions for $x$, and so in each case the subspace $\L$ is not transitive.

\rightline{$\square$}

\section{$\D$-singularity in the particular case (E7)}
\label{E7-singularity}

Let us be given a subspace $\L\in\hbox{Lat}M_3[\C]$ with $\dim\L=5$, assuming that the subspace $\L_0=\L\cap\M_3$ is 2-dimensional.
Let $\D= \hbox{diag}(b_1,b_2,b_3)\in\DD_3$ be arbitrary.
\emph{Let us suppose also that the canonical basis $(L_1,L_2)$ in $\L_0$ is of the form \textup{(E7)}}:
\[L_1=\left[\begin{matrix}
1&0&0\\
x&0&0\\
y&0&0\end{matrix}\right] \quad \hbox{and}\quad L_2= \left[\begin{matrix}
0&1&0\\
b_\D(3)x& x&0\\
b_\D(2)y&y&0\end{matrix}\right] \quad (x, y\in\C).\]
Let $(L_1,L_2,Q_1,Q_2,Q_3)$ be the extended canonical basis in $\L$.
We are going to characterize $\D$-singularity of $\L$ in terms of the entries of the matrices in the canonical basis.

Clearly, if $\L$ is $\D$-singular, then so is every subspace $\L'$ of $\L$.
We shall study subspaces spanned by 3 basis vectors; altogether that means ten 3-dimensional possible subspaces.
Recall that, for every $r\in\N_3,\; \L_r$ denotes the subspace spanned by $\{L_1,L_2,Q_r\}$.
First we characterize $\D$-singularity of these subspaces.

\begin{pro}
\label{Q2Q3-singularity}
Let $\D\in\DD_3$ be arbitrary and suppose that \textup{(E7)} holds.

\medskip
\textup{(a)} Then $\L_2$ is $\D$-singular if and only if

\medskip
\noindent
\textup{(F$7_2$)} $\quad\quad Q_2= \left[\begin{matrix}
0&0&0\\
q_{2,1}^2&q_{2,2}^2&1\\
0&0&0\end{matrix}\right],$

\medskip
\noindent
and $\L_3$ is $\D$-singular if and only if

\medskip
\noindent
\textup{(F$7_3$)} $\quad\quad Q_3= \left[\begin{matrix}
0&0&0\\
0&0&0\\
q_{3,1}^3&q_{3,2}^3&1\end{matrix}\right].$

\medskip
\textup{(b)} If $\L\in\S_3[\D]$ and $xy=0$, then $\L\not\in\T_3$.
\end{pro}

\noindent
{\bf Proof.}
(a): We know by Lemma \ref{basis-completion-5dim} that
\[Q_r=\left[\begin{matrix}
0&0&\d_{1,r}\\
q_{2,1}^r&q_{2,2}^r&\d_{2,r}\\
q_{3,1}^r&q_{3,2}^r&\d_{3,r}\end{matrix}\right] \quad (r\in\N_3),\]
and so the general element of $\L_r$ is of the form
\[X=\left[\begin{matrix}
z_1&z_2& *\\
z_1x+z_2 b_\D(3)x + w_r q_{2,1}^r& z_2x+ w_r q_{2,2}^r&*\\
z_1y + z_2b_\D(2) y +w_r q_{3,1}^r& z_2 y + w_r q_{3,2}^r& *\end{matrix}\right].\]
We have to check when $\L_2$ and $\L_3$ are functionally $\D$-vanishing.

\noindent
For $r=2$, we have
\[F_{b_\D(2)}(X_{P_2})= z_1w_2 q_{3,2}^2 + z_2w_2(-q_{3,1}^2 + b_\D(2) q_{3,2}^2),\]
which is identically zero exactly when $q_{3,2}^2 = q_{3,1}^2=0$, and we obtain (F$7_2$). 

\noindent
For $r=3$, we have
\[F_{b_\D(3)}(X_{P_3})= z_1w_3 q_{2,2}^3 + z_2w_3(-q_{2,1}^3 + b_\D(3) q_{2,2}^3),\]
which is identically zero exactly when $q_{2,2}^3= q_{2,1}^3=0$, and so we obtain (F$7_3$).

(b): If $\L\in\S_3[\D]$, then $\L_2$ and $\L_3$ are functionally $\D$-vanishing, and so (F$7_2$) and (F$7_3$) hold.
Taking $\wt x =\x_1 e_1 + \x_2 e_2 + \x_3 e_3$ and $\wt y= \ol\et_1 e_1 + \ol\et_2 e_2 + \ol\et_3 e_3\ne 0$, let us consider the equations
\begin{itemize}
\item[\textup{(1)}] $\quad 0= \la L_1 \wt x,\wt y\ra = \x_1\et_1+ \x_1x\et_2 + \x_1y\et_3,$
\item[\textup{(2)}] $\quad 0= \la L_2\wt x,\wt y\ra = \x_2 \et_1 +x (b_\D(3)\x_1+\x_2)\et_2 + y (b_\D(2)\x_1+\x_2)\et_3,$
\item[\textup{(3)}] $\quad0= \la Q_1\wt x, \wt y\ra$,
\item[\textup{(4)}] $\quad 0= \la Q_2\wt x,\wt y\ra = (q_{2,1}^2\x_1+ q_{2,2}^2 \x_2 +\x_3)\et_2$,
\item[\textup{(5)}] $\quad 0= \la Q_3\wt x,\wt y\ra = (q_{3,1}^3 \x_1 + q_{3,2}^3\x_2+\x_3)\et_3.$
\end{itemize}
If $x=0$, then choosing $\et_1=\et_3=0$ and $\et_2\ne 0$, the equations (1), (2), (5) evidently hold.
While if $y=0$, then choosing $\et_1=\et_2=0$ and $\et_3\ne 0$, the equations (1), (2), (4) hold evidently.
In both cases the remaining two linear equations have nonzero solution for $\wt x$.

\rightline{$\square$}

\begin{pro}
\label{Q1-singularity}
Let $\D\in\DD_3$ be arbitrary and suppose that $xy\ne0$.
Then $\L_1$ is $\D$-singular if and only if the third element of the canonical basis in $\L$ is of the form

\medskip
\noindent
\textup{(F$7_1$)} $\quad\quad Q_1= \left[\begin{matrix}
0&0&1\\
q_1 & 0&0\\
\frac{y}{x} q_1 & 0&0\end{matrix}\right]\quad (q_1\in\C),$

\medskip

or

\medskip
\noindent
\textup{(F$7_1'$)} $\quad\quad Q_1'= \left[\begin{matrix}
0&0&1\\

q_1'&q_1&0\\

\frac{y}{x} q_1' + \frac{y}{x} b_\D(1)q_1 & \frac{y}{x} q_1&0\end{matrix}\right] \quad (q_1\in\C_0, \; q_1'\in\C).$
\end{pro}

\noindent
{\bf Proof.}
We have to verify again when $\L_1$ is functionally $\D$-vanishing.
We have
\begin{eqnarray*}
F_{b_\D(1)}(X_{P_1}) & = & z_1w_1( q_{3,2}^1 x - q_{2,2}^1 y)\\
& & +z_2w_1\Big( b_\D(3)xq_{3,2}^1 + yq_{2,1}^1 -xq_{3,1}^1 -b_\D(2)yq_{2,2}^1\\
 & & \quad\quad\quad \quad\quad\quad\quad\quad\quad\quad  +b_\D(1)x q_{3,2}^1 + b_\D(1) y q_{2,2}^1\Big)\\
 & &+w_1^2 (q_{2,1}^1 q_{3,2}^1 - q_{2,2}^1 q_{3,1}^1 + b_\D(1) q_{2,2}^1 q_{3,2}^1).\end{eqnarray*}
This polynomial is identically zero exactly when
\[q_{3,2}^1 x= q_{2,2}^1 y,\]
\[ q_{3,2}^1 x b_\D(2) - q_{2,2}^1 y b_\D(3) + y q_{2,1}^1 - x q_{3,1}^1 =0,\]
\[q_{2,2}^1 q_{3,1}^1 = q_{3,2}^1 (q_{2,1}^1 + b_\D(1) q_{2,2}^1).\]
Assuming $q_{2,2}^1 = 0$, we infer that $q_{3,2}^1=0$ and $q_{3,1}^1 = \frac{y}{x} q_{2,1}^1$.
Hence, with $q_{2,1}^1=q_1$ we obtain (F$7_1$).
Assuming $q_{2,2}^1 \ne 0$, we infer that 
\[q_{3,2}^1 = \frac{y}{x} q_{2,2}^1 \quad \hbox{and} \quad q_{3,1}^1 = \frac{y}{x} q_{2,1}^1 + \frac{y}{x} b_\D(1) q_{2,2}^1\]
provide the solution of the three equations above.
Taking $q_{2,2}^1=q_1\in\C_0$ and $q_{2,1}^1= q_1'\in\C$ we obtain (F$7_1'$).

\rightline{$\square$}

Let (E$7_0$) denote the form (E7) together with the assumption $xy\ne0$.
Furthermore, let (F7) stand for the canonical basis of $\L$, where (E$7_0$), (F$7_1$), (F$7_2$), (F$7_3$) hold.
Similarly, let (F$7'$) denote the canonical basis of $\L$, where (E$7_0$), (F$7_1'$), (F$7_2$), (F$7_3$) hold.

First we consider the case when (F7) holds, giving a complete characterization of $\D$-singularity.

\begin{thm}
\label{F7-singularity}
Let $\D= \hbox{\rm diag}(b_1,b_2,b_3)\in\DD_3$ be arbitrary, and suppose that \textup{(F7)} holds.
Then $\L$ is $\D$-singular if and only if
\begin{itemize}
\item[\textup{(i)}] $\quad q_{2,2}^2= b_\D(3),$
\item[\textup{(ii)}] $\quad q_{2,1}^2 = b_\D(3)/b_2,$
\item[\textup{(iii)}] $\quad q_{3,2}^3 = b_\D(2),$
\item[\textup{(iv)}] $\quad q_{3,1}^3 = b_\D(2)/b_3,$
\item[\textup{(v)}] $\quad q_1=0.$
\end{itemize}
\end{thm}

\noindent
{\bf Proof.}
First we verify the necessity of the conditions (i)--(v).
So let us assume that $\L$ is $\D$-singular.

The linear combination $X= z_1L_1 + w_1 Q_1 +w_2Q_2$ is of the form
\[\left[\begin{matrix}
z_1&0&w_1\\
z_1x+w_1q_1 + w_2 q_{2,1}^2& w_2 q_{2,2}^2&w_2\\
z_1y+w_1\frac{y}{x}q_1&0&0\end{matrix}\right],\]
whence
\[\G_\D(X) = \left[\begin{matrix}
z_1b_1 & -w_1& b_1w_1\\
z_1b_2x + w_1b_2q_1 +w_2(b_2q_{2,1}^2-q_{2,2}^2) & w_2(b_2q_{2,2}^2-1) & b_2w_2\\
z_1b_3y+w_1b_3\frac{y}{x}q_1&0&0\end{matrix}\right]\]
follows; see Section \ref{subspaces-of-matrices}.
Expending the determinant along the third row we obtain
\[\det\G_\D(X) = -z_1w_1w_2 b_3 y(b_2-b_1+b_1b_2 q_{2,2}^2) - w_1^2 w_2 b_3\frac{y}{x} q_1 (b_2-b_1+b_1b_2q_{2,2}^2).\]
This polynomial is zero if and only if $b_2-b_1+b_1b_2q_{2,2}^2=0$, that is when $q_{2,2}^2= b_\D(3)$.

The linear combination $X= z_2L_2 +w_1Q_1+w_2Q_2$ is
\[X= \left[\begin{matrix}
0&z_2&w_1\\
z_2x b_\D(3)+ w_1q_1+w_2q_{2,1}^2& z_2x+ w_2 q_{2,2}^2& w_2\\
z_2yb_\D(2)+ w_1\frac{y}{x}q_1& z_2y&0\end{matrix}\right],\]
and so $\G_\D(X)$ is of the form
\[\left[\begin{matrix}
-z_2& z_2b_1-w_1& b_1w_1\\
z_2x(b_2 b_\D(3)-1)+ w_1b_2q_1+ w_2(b_2q_{2,1}^2-q_{2,2}^2)& z_2b_2x + w_2(b_2q_{2,2}^2-1)& b_2w_2\\
z_2y(b_3b_\D(2)-1) +w_1b_3\frac{y}{x}q_1& z_2yb_3&0\end{matrix}\right].\]
Applying Sarrus' rule we infer that the determinant is of the form
\[\det\G_\D(X)= A z_2^2w_1 + Bz_2^2w_2 + Cz_2w_1^2 + Dz_2w_1w_2 + Ew_1^2w_2.\]
All coefficients here are equal to zero.
In particular, we have
\[D= b_1b_2b_3 y q_{2,1}^2 -(b_1b_3-b_2b_3)yq_{2,2}^2 +b_1b_2b_3\frac{y}{x}q_1+\frac{b_2b_3}{b_1}y -b_3y.\]
Thus, $D=0$ and $q_{2,2}^2=b_\D(3)$ imply
\[q_{2,1}^2= b_\D(3)q_{2,2}^2-\frac{q_1}{x} + \frac{b_\D(3)}{b_1}= \frac{b_\D(3)}{b_2}-\frac{q_1}{x}.\]

The linear combination $X=z_1L_1 + w_1Q_1+w_3Q_3$ is
\[X= \left[\begin{matrix}
z_1&0&w_1\\
z_1x+w_1q_1& 0&0\\
z_1y+ w_1\frac{y}{x}q_1+ w_3q_{3,1}^3&w_3q_{3,2}^3& w_3\end{matrix}\right],\]
whence
\[\G_\D(X)= \left[\begin{matrix}
z_1b_1& -w_1& b_1w_1\\
z_1b_2x + w_1b_2q_1& 0&0\\
z_1b_3y+ w_1b_3\frac{y}{x}q_1 + w_3(b_3q_{3,1}^3-1)& w_3(b_3q_{3,2}^3-1)& b_3w_3\end{matrix}\right].\]
Expanding the determinant along the second row, we obtain
\[\det\G_\D(X)= z_1w_1w_3b_2x(b_3-b_1+b_1b_3q_{3,2}^3) + w_1^2w_3 b_2q_1(b_3-b_1+b_1b_3 q_{3,2}^3).\]
This polynomial is identically zero exactly when $b_3-b_1+b_1b_3 q_{3,2}^3=0$, that is when $q_{3,2}^3= b_\D(2)$.

The linear combination $X=z_2L_2+w_1Q_1+w_3Q_3$ is 
\[X=\left[\begin{matrix}
0&0&w_1\\
z_2xb_\D(3)+w_1q_1& z_2x& 0\\
z_2yb_\D(2)+w_1\frac{y}{x}q_1+w_3q_{3,1}^3& z_2y+ w_3q_{3,2}^3& w_3\end{matrix}\right],\]
hence $\G_\D(X)$ is of the form
\[\left[\begin{matrix}
0&-w_1&b_1w_1\\
z_2x (b_2b_\D(3)-1)+ w_1b_2q_1 & z_2b_2x&0\\
z_2y(b_3b_\D(2)-1)+w_1\frac{y}{x}b_3q_1+w_3(b_3q_{3,1}^3-q_{3,2}^3)& z_2b_3 y+ w_3(b_3q_{3,2}^3-1)& b_3w_3\end{matrix}\right].\]
By Sarrus' rule we infer
\[\det\G_\D(X)= A z_2^2w_1 + B z_2w_1^2 + C z_2w_1w_3 + D w_1^2w_3.\]
A short computation yields 
\[C=-b_1b_2b_3xq_{3,1}^3 +(b_1b_2-b_2b_3)x q_{3,2}^3 +b_2x -\frac{b_2b_3}{b_1}x.\]
Since $C=0$ and $q_{3,2}^3=b_\D(2)$, it follows that
\[q_{3,1}^3 = b_\D(2)q_{3,2}^3 + \frac{b_\D(2)}{b_1} = \frac{b_\D(2)}{b_3}.\]

The linear combination $X= z_2L_2+ w_2Q_2+w_3Q_3$ is
\[X= \left[\begin{matrix}
0&z_2&0\\
z_2xb_\D(3)+ w_2 q_{2,1}^2& z_2x + w_2 q_{2,2}^2 & w_2\\
z_2y b_\D(2) + w_3q_{3,1}^3& z_2y +w_3q_{3,2}^3& w_3\end{matrix}\right],\]
hence $\G_\D(X)$ is of the form
\[\left[\begin{matrix}
-z_2& z_2b_1&0\\
z_2x (b_2b_\D(3)-1)+ w_2(b_2q_{2,1}^2-q_{2,2}^2)& z_2xb_2 + w_2(b_2q_{2,2}^2-1)& w_2b_2\\
z_2y(b_3b_\D(2)-1) + w_3(b_3q_{3,1}^3-q_{3,2}^3)& z_2yb_3 + w_3 (b_3q_{3,2}^3-1) & w_3b_3\end{matrix}\right].\]
We have \[\det\G_\D(X)= Az_2^2w_2 + B z_2^2w_3 + Cz_2w_2w_3,\]
where
\[C= -(b_2q_{2,2}^2-1)b_3 + b_1b_2(b_3q_{3,1}^3-q_{3,2}^3) - b_1(b_2q_{2,1}^2 -q_{2,2}^2)b_3 + b_2(b_3q_{3,2}^3-1).\]
Substituting the values in (i)--(iv), we obtain that $C= b_1b_2b_3 q_1/x$.
Thus, $C=0$ yields that $q_1=0$.

We proceed with the proof of sufficiency.
So let us assume that (i)--(v) hold.
Then the general element $X=z_1L_1 + z_2L_2 +w_1Q_1 +w_2Q_2+ w_3Q_3$ of $\L$ is of the form
\[\left[\begin{matrix}
z_1&z_2&w_1\\
z_1x + z_2b_\D(3)x + w_2 b_\D(3)/b_2& z_2x +w_2 b_\D(3)& w_2\\
z_1y+ z_2b_\D(2)y+ w_3 b_\D(2)/b_3& z_2y+ w_3b_\D(2)& w_3\end{matrix}\right].\]
Hence
\[\G_\D(X)= \left[\begin{matrix}
z_1b_1-z_2& z_2b_1-w_1& b_1w_1\\
z_1b_2x-z_2xb_2/b_1& z_2xb_2 -w_2b_2/b_1 & b_2w_2\\
z_1b_3y-z_2yb_3/b_1& z_2yb_3 - w_3b_3/b_1& b_3w_3\end{matrix}\right].\]
Taking into account that
\[z_1b_2x-z_2x\frac{b_2}{b_1} = \frac{b_2}{b_1}x (z_1b_1-z_2), \quad z_1b_3y-z_2y\frac{b_3}{b_1} = \frac{b_3}{b_1}y (z_1b_1-z_2),\]
we obtain that
\[\det\G_\D(X)= (z_1b_1-z_2) \det Y,\]
where
\[ Y= \left[\begin{matrix}
1& z_2b_1-w_1& b_1w_1\\
b_2x/b_1& z_2xb_2-w_2b_2/b_1& b_2w_2\\
b_3y/b_1& z_2y b_3 -w_3b_3/b_1& b_3w_3\end{matrix}\right].\]
Application of Sarrus' rule yields after a short computation that $\det Y\equiv 0$, and so $\det\G_\D(X)\equiv 0$.

\rightline{$\square$}

Now we turn to the characterization of $\D$-singularity of the form (F$7'$).

\begin{thm}
\label{F7comma-singularity}
Let $\D= \hbox{\rm diag}(b_1,b_2,b_3)\in\DD_3$ be arbitrary, and let us assume that \textup{(F$7'$)} holds.
Then $\L$ is $\D$-singular if and only if
\begin{itemize}
\item[\textup{(i)}] $\quad q_{2,2}^2= b_\D(3) - q_1/x$,
\item[\textup{(ii)}] $\quad q_{2,1}^2 = b_\D(3)/b_2 - q_1'/x,$
\item[\textup{(iii)}] $\quad q_{3,2}^3 =b_\D(2) - q_1/x,$
\item[\textup{(iv)}] $\quad q_{3,1}^3 = b_\D(2)/b_3 - b_\D(1) q_1/x - q_1'/x.$
\end{itemize}
\end{thm}

\noindent
{\bf Proof.}
First we prove necessity.
So let us assume that the subspace $\L$ is $\D$-singular.

Let us consider the linear combination $X=z_1L_1 +w_1Q_1+w_2Q_2$:
\[X=\left[\begin{matrix}
z_1& 0& w_1\\
z_1x +w_1q_1' +w_2q_{2,1}^2& w_1q_1 +w_2 q_{2,2}^2 & w_2\\
z_1y + w_1\frac{y}{x} \wt q  & w_1\frac{y}{x}q_1& 0\end{matrix}\right]\quad \hbox{where}\;\; \wt q=q_1' + b_\D(1)q_1.\]
Then $\G_\D(X)$ is of the form
\[\left[\begin{matrix}
z_1b_1 & -w_1& w_1b_1\\
z_1b_2x+ w_1(b_2q_1'-q_1)+ w_2(b_2q_{2,1}^2-q_{2,2}^2)& w_1b_2q_1 + w_2(b_2q_{2,2}^2-1)& w_2b_2\\
z_1b_3y + w_1\frac{y}{x}(b_3\wt q-q_1)& w_1\frac{y}{x} b_3q_1& 0\end{matrix}\right],\]
hence
\[\det\G_\D(X)= A z_1w_1w_2 + B z_1w_1^2 + C w_1^2w_2 + D w_1^3.\]
The coefficient $A$ is:
\[A= -b_1b_2b_3yq_{2,2}^2 + (b_1b_3-b_2b_3)y -b_1b_2b_3y \frac{q_1}{x}.\]
Since $A=0$, it follows that
\[q_{2,2}^2 = b_\D(3) -\frac{q_1}{x}.\]
Using this value of $q_{2,2}^2$, straightforward computation yields that
\[C= \frac{y}{x} b_1b_2b_3 q_1 \left(q_{2,1}^2 - \frac{b_\D(3)}{b_2} + \frac{q_1'}{x}\right).\]
Since $C=0$, we conclude that
\[q_{2,1}^2 = \frac{b_\D(3)}{b_2} - \frac{q_1'}{x}.\]

Let us consider the linear combination $X=z_1L_1+w_1Q_1+w_3Q_3$:
\[X=\left[\begin{matrix}
z_1& 0& w_1\\
z_1x+w_1q_1'& w_1q_1& 0\\
z_1y +w_1\frac{y}{x} \wt q+ w_3q_{3,1}^3 & w_1\frac{y}{x} q_1 +w_3 q_{3,2}^3& w_3\end{matrix}\right].\]
Then $\G_\D(X)$ is of the form
\[\left[\begin{matrix}
z_1b_1 & -w_1 & w_1b_1\\
z_1b_2x + w_1(b_2q_1'-q_1)& w_1b_2q_1& 0\\
z_1b_3y +w_1\frac{y}{x} (b_3\wt q-q_1) + w_3 (b_3q_{3,1}^3 -q_{3,2}^3)& w_1\frac{y}{x} b_3q_1 + w_3(b_3q_{3,2}^3-1)& w_3b_3\end{matrix}\right],\]
hence
\[\det\G_\D(X)= Az_1w_1w_3 + Bz_1w_1^2 +C w_1^2w_3 + Dw_1^3.\]
It is easy to check that
\[A= b_1b_2b_3xq_{3,2}^3 - b_1b_2x + b_2b_3 x +b_1b_2b_3 q_1.\]
Since $A=0$, we obtain
\[q_{3,2}^3 = b_\D(2) - \frac{q_1}{x}.\]
Using this value, straightforward computation yields
\[ C= -b_1b_2b_3q_1 q_{3,1}^3 +\left(\frac{b_1b_2}{b_3} -b_2\right)q_1 - (b_1b_2-b_1b_3)\frac{q_1^2}{x} - b_1b_2b_3 \frac{q_1'q_1}{x}.\]
Since $C=0$, it follows that
\[q_{3,1}^3 = \frac{b_\D(2)}{b_3} - b_\D(1) \frac{q_1}{x} - \frac{q_1'}{x}.\]

Let us proceed with the proof of sufficiency.
So let us assume that the conditions (i)--(iv) hold.
Let us consider the general element $X= z_1L_1 +z_2L_2 +w_1Q_1'+w_2Q_2+w_3 Q_3$ of $\L$:
\[\left[\begin{matrix}z_1&z_2&w_1\\
z_1x+ z_2b_\D(3)x + w_1q_{2,1}^1+w_2q_{2,1}^2& z_2x+w_1q_{2,2}^1+w_2 q_{2,2}^2& w_2\\
z_1y+ z_2b_\D(2)y+w_1q_{3,1}^1+w_3 q_{3,1}^3& z_2y+ w_1q_{3,2}^1 +w_3 q_{3,2}^3& w_3\end{matrix}\right],\]
where
\[q_{2,1}^1=q_1', \; q_{2,2}^1=q_1, \; q_{3,1}^1= \frac{y}{x}q_1'+ \frac{y}{x} b_\D(1)q_1, \; q_{3,2}^1 = \frac{y}{x} q_1,\]
and (i)--(iv) are valid for the entries of $Q_2$ and $Q_3$.
Then $\G_\D(X)$ is the matrix $\wt X= [\wt\x_{i,j}]_3$ with the entries
\begin{eqnarray*}
\wt\x_{1,1} & = & z_1b_1-z_2,\\
\wt\x_{2,1}& = & (z_1b_1-z_2)x\frac{b_2}{b_1} - w_1(q_1-b_2q_1') + w_2 \frac{1}{x}(q_1-b_2q_1'),\\
\wt\x_{3,1}&=& (z_1b_1-z_2)y \frac{b_3}{b_1} - w_1\frac{y}{x} \frac{b_3}{b_2} (q_1-b_2q_1') + w_3 \frac{1}{x} \frac{b_3}{b_2} (q_1-b_2q_1');\\
\wt\x_{1,2}&=& z_2b_1-w_1,\\
\wt\x_{2,2}&=& z_2xb_2+ w_1b_2q_1 - w_2\frac{b_2}{b_1} \left(q_1\frac{b_1}{x}+1\right),\\
\wt\x_{3,2}&=& z_2yb_3 + w_1\frac{y}{x}b_3 q_1 -w_3 \frac{b_3}{b_1}\left(q_1\frac{b_1}{x} +1\right);
\end{eqnarray*}
\[\wt\x_{1,3}= w_1b_1,\quad \wt\x_{2,3} = w_2b_2, \quad \wt\x_{3,3}= w_3b_3.\]
Applying basic rules concerning determinants we can see that $\det\wt X$ splits into the sum
\[\det\wt X= (z_1b_1-z_2) \det Y + z_2 \det Z + \det V,\]
where
\[Y= \left[\begin{matrix}
1&z_2b_1-w_1& w_1b_1\\
xb_2/b_1 & z_2xb_2+w_1b_2q_1-w_2b_2(q_1b_1/x+1)/b_1&w_2b_2\\
yb_3/b_1& z_3yb_3+w_1b_3q_1y/x-w_3b_3(q_1b_1/x+1)/b_1&w_3b_3\end{matrix}\right],\]
\[Z=\left[\begin{matrix}
0&b_1&w_1b_1\\
-w_1(q_1-b_2q_1')+ w_2(q_1-b_2q_1')/x&xb_2&w_2b_2\\
-w_1(q_1-b_2q_1')yb_3/(xb_2)+ w_3(q_1-b_2q_1') b_3/(xb_2)& yb_3&w_3b_3\end{matrix}\right],\]
and
\[V=\left[\begin{matrix}
0&-w_1& w_1b_1\\
v_{2,1}&v_{2,2}&w_2b_2\\
v_{3,1}&v_{3,2}&w_3b_3\end{matrix}\right]\]
with
\begin{eqnarray*}
v_{2,1}&=&-w_1(q_1-b_2q_1')+w_2(q_1-b_2q_1')/x,\\
v_{2,2}&=& w_1b_2q_1-w_2(q_1b_1/x+1)b_2/b_1,\\
v_{3,1}&=&-w_1(q_1-b_2q_1') yb_3/(xb_2)+ w_3(q_1-b_2q_1') b_3/(xb_2),\\
v_{3,2}&=& w_1b_3q_1y/x - w_3(q_1b_1/x+1)b_3/b_1.
\end{eqnarray*}
We infer
\[\det Y= \frac{1}{b_1} \det\left[\begin{matrix}
b_1&-w_1&w_1b_1\\
xb_2& w_1b_2q_1-w_2(q_1b_1/x+1)b_2/b_1&w_2b_2\\
yb_3& w_1b_3q_1y/x-w_3(q_1b_1/x+1)b_3/b_1&w_3b_3\end{matrix}\right]\]
\begin{eqnarray*} &=&\frac{1}{b_1} w_1\det\left[\begin{matrix}
b_1&-1&w_1b_1\\
xb_2& b_2q_1&w_2b_2\\
yb_3& b_3q_1y/x& w_3b_3\end{matrix}\right]\\
 & &-\frac{1}{b_1^2} \left(q_1\frac{b_1}{x}+1\right) \det\left[\begin{matrix}
b_1&0&w_1b_1\\
xb_2&w_2b_2&w_2b_2\\
yb_3&w_3b_3&w_3b_3\end{matrix}\right].
\end{eqnarray*}
Applying Laplace expansion along the third column and along the first row, respectively, we obtain that $\det Y=0$.

Similarly, we have
\begin{eqnarray*}
\det Z&=& b_1(q_1-b_2q_1') \det\left[\begin{matrix}
0&1&w_1\\
-w_1+w_2/x& xb_2&w_2b_2\\
-w_1yb_3/(xb_2)+ w_3 b_3/(xb_2)&yb_3&w_3b_3\end{matrix}\right]\\
 &=& b_1(q_1-b_2q_1') \det\left[\begin{matrix}
0&1&0\\
-w_1+w_2/x& xb_2& w_2b_2-w_1xb_2\\
-w_1 yb_3/(xb_2)+w_3b_3/(xb_2)& y b_3& w_3b_3-w_1yb_3\end{matrix}\right],
\end{eqnarray*}
which yields after some computation that $\det Z=0$.

Finally, the expression

\medskip
\noindent
$\;\;\det V= w_1(q_1-b_2q_1') \times$

\medskip
$\quad
 \times\det\left[\begin{matrix}
0&-1&0\\
-w_1+w_2/x& *& w_1b_1b_2q_1 - w_2 q_1b_1b_2/x\\
-w_1yb_3/(xb_2) + w_3 b_3/(xb_2)&*&w_1b_1b_3q_1y/x-w_3 q_1b_1b_3/x\end{matrix}\right]$

\medskip
\noindent
results in that $\det V=0$.

Therefore, $\det\wt X=0$ and so the subspace $\L$ is $\D$-singular.

\rightline{$\square$}

We say that \emph{the canonical basis of $\L$ is of the form \textup{(G7)}}, if it is of the form (F7), satisfying also the conditions (i)--(v) of Theorem \ref{F7-singularity}.
Similarly, \emph{the canonical basis is of the form \textup{(G$7'$)}}, if it is of the form (F$7'$), satisfying the conditions (i)--(iv) of Theorem \ref{F7comma-singularity}.

\section{Consistency with transitivity}
\label{consistancy}

We are going to show that if (G7) or (G$7'$) holds, then the subspace is transitive.

\begin{pro}
\label{G7-transitivity}
Let $\D= \hbox{\textup{diag}}(b_1,b_2,b_3)\in\DD_3$ be arbitrary.
If \textup{(G7)} holds, then $\L$ is transitive.
\end{pro}

\noindent
{\bf Proof.}
We have to examine the system of equations
\begin{itemize}
\item[\textup{(1)}] $\quad 0= \x_1(\et_1+x\et_2+y\et_3),$
\item[\textup{(2)}] $\quad 0 = \x_1(b_\D(3)x\et_2+ b_\D(2)y\et_3) + \x_2(\et_1+ x\et_2+y\et_3),$
\item[\textup{(3)}] $\quad 0= \x_3\et_1,$
\item[\textup{(4)}] $\quad 0= \left(\frac{b_\D(3)}{b_2} \x_1 + b_\D(3)\x_2 + \x_3\right) \et_2,$
\item[\textup{(5)}] $\quad 0= \left(\frac{b_\D(2)}{b_3}\x_1+ b_\D(2)\x_2+ \x_3\right)\et_3$.
\end{itemize}
We are going to show that given any $(\et_1,\et_2,\et_3)\ne(0,0,0)$, this system has only the trivial $(0,0,0)$ solution for $(\x_1,\x_2,\x_3)$.
Recall that $xy\ne0$.

If $\et_1+x\et_2+y\et_3\ne 0$, then (1) and (2) imply $\x_1=\x_2=0$.
Hence (3), (4), (5) yield that $\x_3=0$.

Suppose from now on that $\et_1+x\et_2+ y\et_3=0$.
Then (1) is automatic and (2) takes the form 
\[\x_1(b_\D(3)x\et_2 + b_\D(2)y\et_3)=0.\]
If $b_\D(3)x\et_2+ b_\D(2)y\et_3\ne 0$, then $\x_1=0$.
Assuming also $\et_1\ne0$, (3) yields $\x_3=0$.
Since $\et_1+x\et_2+y\et_3=0$, it follows that $\et_2\ne 0$ or $\et_3\ne0$, and so (4) or (5) implies $\x_2=0$.
If $\et_1=0$, then (3) is automatic.
Furthermore $\et_2= -\et_3y/x$ cannot be zero.
Thus (4) and (5) take the form
\[b_\D(3)\x_2+\x_3=0, \quad b_\D(2)\x_2+\x_3=0.\]
Since $b_\D(2)\ne b_\D(3)$, it follows that $\x_2=\x_3=0$.

Finally, let us assume that
\[\et_1+x\et_2+y\et_3=0 \quad\hbox{and}\quad b_\D(3)x\et_2+ b_\D(2)y\et_3=0.\]
Then (1) and (2) evidently hold.
Furthermore
\[\et_2=-\frac{b_\D(2)}{b_\D(3)} \frac{y}{x}\et_3 \quad \hbox{and} \quad  \et_1= -x\et_2-y\et_3= \frac{b_\D(1)}{b_\D(3)}y\et_3\]
are nonzero.
Hence (3) yields $\x_3=0$, and so (4) and (5) reduce to
\[\x_1+ b_2\x_2=0 \quad \hbox{and} \quad \x_1 + b_3\x_2=0.\]
Since $b_2\ne b_3$, it follows that $\x_1=\x_2=0.$

\rightline{$\square$}

\begin{pro}
\label{G7comma-transitivity}
Let $\D= \hbox{\textup{diag}}(b_1,b_2,b_3)\in\DD_3$ be arbitrary.
If \textup{(G$7'$)} holds, then $\L$ is transitive.
\end{pro}

\noindent
{\bf Proof.}
Now the system of equations we have to study is the following:
\begin{itemize}
\item[\textup{(1)}] $\quad 0= \x_1(\et_1+x\et_2+y\et_3)$,
\item[\textup{(2)}] $\quad 0= \x_1(b_\D(3)x\et_2 + b_\D(2)y\et_3) + \x_2(\et_1+x\et_2+y\et_3)$,
\item[\textup{($3'$)}] $\quad 0= (q_1'\x_1+q_1\x_2)\et_2 + \frac{y}{x}\big((q_1'+b_\D(1)q_1)\x_1+q_1\x_2\big)\et_3+ \x_3\et_1$,
\item[\textup{(4)}] $\quad 0=\left(\left(\frac{b_\D(3)}{b_2}-\frac{q_1'}{x}\right)\x_1+ \left(b_\D(3)-\frac{q_1}{x}\right)\x_2 +\x_3\right)\et_2$,
\item[\textup{(5)}] $\quad 0= \left( \left(\frac{b_\D(2)}{b_3} -b_\D(1) \frac{q_1}{x}-\frac{q_1'}{x}\right)\x_1 + \left(b_\D(2)-\frac{q_1}{x}\right)\x_2 +\x_3\right)\et_3$.
\end{itemize}
We assume that $(\et_1,\et_2,\et_3)\ne (0,0,0)$.

If $\et_1+x\et_2+y\et_3\ne0$, then (1) and (2) yield $\x_1=\x_2=0$.
Since $\et_i\ne0$ is true for some $i\in\N_3$, one of the equations ($3'$), (4), (5) implies that $\x_3=0$.

Suppose that
\[\et_1+x\et_2+y\et_3=0 \quad \hbox{and} \quad b_\D(3)x\et_2 + b_\D(2)y\et_3\ne0.\]
Then (1) is valid trivially, and (2) yields $\x_1=0$.
Thus, the remaining three equations take the form
\begin{itemize}
\item[\textup{($3'$)}] $\quad 0= (q_1\x_2-x\x_3)\et_2 + \left(\frac{y}{x}q_1\x_2-y\x_3\right)\et_3$,
\item[\textup{(4)}] $\quad 0= \left(\left(b_\D(3)-\frac{q_1}{x}\right)\x_2 +\x_3\right)\et_2$,
\item[\textup{(5)}] $\quad 0= \left(\left(b_\D(2)-\frac{q_1}{x}\right)\x_2+\x_3\right)\et_3.$
\end{itemize}
If $\et_2=0\ne\et_3$, then this system reduces to
\begin{itemize}
\item[\textup{($3'$)}] $\quad 0= \frac{q_1}{x} \x_2 -\x_3$,
\item[\textup{(5)}] $\quad 0= \left(b_\D(2)- \frac{q_1}{x}\right)\x_2+\x_3,$
\end{itemize}
which has only the zero solution $\x_2=\x_3=0$.
The cases $\et_2\ne0=\et_3$ and $\et_2\ne0\ne\et_3$ can be treated similarly, leading to the trivial solution $\x_2=\x_3=0$.

Finally, suppose that
\[\et_1+x\et_2+y\et_3=0 \quad \hbox{and} \quad b_\D(3)x\et_2+b_\D(2) y\et_3=0.\]
Then we have
\[\et_2 = - \frac{b_\D(2)}{b_\D(3)} \frac{y}{x}\et_3 \quad \hbox{and} \quad \et_1= \frac{b_1(1)}{b_\D(3)} y\et_3.\]
It follows that $\et_3\ne0$, and so $\et_1\ne0\ne\et_2$ also hold.
Hence $\et_2$ and $\et_3$ can be cancelled in (4) and (5), respectively.
Furthermore, substituting the previous expressions of $\et_1$ and $\et_2$ into ($3'$), this equation can be written in the form
\begin{itemize}
\item[\textup{($3'$)}] $\quad 0= \left(\frac{q_1'}{x} - b_\D(3) \frac{q_1}{x}\right) \x_1 + \frac{q_1}{x}\x_2 -\x_3$.
\end{itemize}
The coefficient matrix of the system ($3'$), (4), (5) is
\[C=\left[\begin{matrix}
q_1'/x -b_\D(3)q_1/x & q_1/x& -1\\
b_\D(3)/b_2-q_1'/x & b_\D(3)-q_1/x& 1\\
b_\D(2)/b_3-b_\D(1)q_1/x-q_1'/x& b_\D(2)- q_1/x&1\end{matrix}\right].\]
We can easily evaluate its determinant:
\begin{eqnarray*}
\det C &=& \det\left[\begin{matrix}
q_1'/x-b_\D(3)q_1/x & q_1/x& -1\\
b_\D(3)/b_2-b_\D(3)q_1/x& b_\D(3)& 0\\
b_\D(2)/b_3-b_\D(2)q_1/x& b_\D(2)&0\end{matrix}\right]\\
 &= & -b_\D(2)b_\D(3) \det\left[\begin{matrix}
1/b_2-q_1/x & 1\\
1/b_3-q_1/x& 1\end{matrix}\right]\\
 &=& b_\D(2)b_\D(3) \left(1/b_3-1/b_2\right).
\end{eqnarray*}
Since $b_2\ne b_3$, it follows that $\det C\ne0$, and so the system has only the zero solution $(\x_1,\x_2,\x_3)= (0,0,0)$.

\rightline{$\square$}

\section{Conclusions}
\label{conclusions}

Now we summarize the results we have achieved up till now.
Let us be given a subspace $\L\in \hbox{Lat} M_3[\C]$ and a matrix $\D\in\DD_3$.
We recall that
\[\L_0=\L\cap\M_3= \left\{X=[\x_{i,j}]_3\in\L: \x_{1,3}=\x_{2,3}=\x_{3,3}=0\right\}.\]
Notice that $\dim\L_0\ge2$, when $\dim\L=5$.
We are going to collect the conditions resulting that $\L$ is not transitive.

We can make distinctions on the basis of dimension.
Recall that $\dim\L\le 6$ is true, whenever $\L\in\S_3[\D]$ (see Proposition \ref{possible-dimension}).

\begin{pro}
\label{dimension-condition}
Suppose that $\L\in\S_3[\D]$.
Then $\L$ is not transitive in any of the following cases:
\begin{itemize}
\item[\textup{(D1)}] $\quad \dim\L=6,$
\item[\textup{(D2)}] $\quad \dim\L<5$,
\item[\textup{(D3)}] $\quad \dim\L=5$ and $\dim \L_0>2$.
\end{itemize}
\end{pro}

\noindent
{\bf Proof.}
If $\dim\L=6$, then Theorem \ref{regularity-6dimension} states the result.
While if $\dim\L<5$, then Proposition \ref{possible-dimension} can be applied.
In connection with (D3) we refer to the proof of Lemma \ref{3dimensional-L0}.

\rightline{$\square$}

It remains to study the case, when
\begin{itemize}
\item[\textup{(D4)}] $\quad\dim\L=5$ and $\dim\L_0=2,$
\end{itemize}
what we shall assume in the sequel.

We can make distinctions also on the basis of the properties of $\L_0$.
In view of Proposition \ref{regularity-condition} we know that if $\L\in\T_3$ and $\L_0$ is not functionally $\D$-vanishing, then $\L\in\R_3[\D]$.
Equivalently, $\L\in\S_3[\D]$ implies that $\L\not\in \T_3$ or $\L_0$ is functionally $\D$-vanishing.
Consequently, we can state that the conditions $\L\in\S_3[\D]$ and that $\L_0$ is not functionally $\D$-vanishing imply $\L\not\in\T_3$.
The canonical basis $(L_1, L_2)$ of $\L_0$ plays a significant role here.
In Section \ref{5dim-transitive} we have proved that $\L_0$ is functionally $\D$-vanishing if and only if $(L_1,L_2)$ is of the form (En) for some $n\in\N_{14}$.
Furthermore, it turned out in Section \ref{exceptional-forms} that if (En) holds with an $n\in\N_{14}\setminus\{7\}$, then $\L\not\in\T_3$.
It has been shown also that if (E7) holds but (E$7_0$) fails, then $\L\not\in\T_3$; see Proposition \ref{Q2Q3-singularity}.
Consequently, we obtain

\begin{pro}
\label{E7-sufficiency}
Suppose that $\L\in\S_3[\D]$ and \textup{(D4)} hold.
If $(L_1,L_2)$ is not of the form \textup{(E$7_0$)}, then $\L\not\in\T_3$.
\end{pro}
We recall that (E$7_0$) stands for the form
\[L_1=\left[\begin{matrix}
1&0&0\\
x&0&0\\
y&0&0\end{matrix}\right] \quad \hbox{and}\quad L_2= \left[\begin{matrix}
0&1&0\\
b_\D(3)x& x&0\\
b_\D(2)y&y&0\end{matrix}\right], \quad \hbox{with}\; xy\not= 0.\]
Thus, we have reduced our investigation to the situation, when (E$7_0$) holds.
The $\D$-singularity has been characterized in Section \ref{E7-singularity} in terms of the extended canonical basis $(L_1,L_2,Q_1,Q_2,Q_3)$ of $\L$.
Namely, we have
\begin{eqnarray*}\L\in\S_3[\D] & \imp & \L_1\in\S_3[\D] \iff \textup 
{(F$7_1$)} \; \hbox{or}\;  \textup{(F$7_1'$)} \;\hbox{holds for}\; Q_1,\\
 & \imp & \L_2\in\S_3[\D] \iff \textup{(F$7_2$)} \; \hbox{holds for}\; Q_2,\\
 &\imp &\L_3\in\S_3[\D] \iff \textup{(F$7_3$)} \; \hbox{holds for}\; Q_3.
\end{eqnarray*}
Then (F7) means that (E$7_0$), (F$7_1$), (F$7_2$), (F$7_3$) hold; similarly (F$7'$) means that (E$7_0$), (F$7_1'$), (F$7_2$), (F$7_3$) hold.
We have concluded in Sections \ref{E7-singularity} and \ref{consistancy} that if (F7) is valid, then
\[\L\in\S_3[\D] \iff \textup{(G7)}\; \hbox{holds} \; \imp \L \; \hbox{is transitive};\]
while if (F$7'$) is valid, then
\[\L\in\S_3[\D] \iff \textup{(G$7'$)}\; \hbox{holds} \;\imp \L \; \hbox{is transitive}.\]
Contrasting these statements with Proposition \ref{E7-sufficiency}, we obtain

\begin{thm}
\label{E7-characterization}
Suppose that $\L\in\S_3[\D]$ and \textup{(D4)} hold.
Then $\L\not\in\T_3$ is true if and only if $(L_1,L_2)$ is not of the form \textup{(E$7_0$)}.
\end{thm}

Now we turn to the operator
\[T=\left[\begin{matrix}
A&C\\
0&B\end{matrix}\right]\in \L(\HH_1\op\HH_2),\]
where $A$ is similar to the unilateral shift $S$ and $\dim\HH_2=\aleph_0$.
Let us consider the similarity model $\wh T$ of $T$ constructed in Section \ref{sim-model} and the cross-sections $\L_{*,r}\; (r\in\N)$ of the commutant of $\wh T$ introduced in Section \ref{cross-sections}.
In view of Proposition \ref{cross-section-transitivity}, we can derive from Proposition \ref{dimension-condition} and Theorem \ref{E7-characterization} the following

\begin{thm}
\label{main-result}
The operator $T$, given above, has a nontrivial hyperinvariant subspace in the following cases:
\begin{itemize}
\item[\textup{(i)}]     \textup{(D1), (D2)} or \textup{(D3)} holds, for some $\L_{*,r}\; (r\in\N)$;
\item[\textup{(ii)}]  \textup{(D4)} holds for every $r\in\N$, and the canonical basis $(L_1,L_2)$ in $\L_{*,r}\cap\M_3$ is not of the form \textup{(E$7_0$)} for some $r\in\N$.
\end{itemize}
More precisely, in these cases there exists $0\ne x\in\HH_1$ such that $\{T\}'x$ is not dense in $\HH_1\op\HH_2$, and so its closure is a proper hyperinvariant subspace of $T$.
\end{thm}

Let us examine also the situation, when the subspace $\HH_2$ above is finite dimensional.

\begin{pro}
\label{H2-finite}
If $\dim\HH_2<\aleph_0$, then $T$ has a nontrivial hyperinvariant subspace.
\end{pro}

\noindent
{\bf Proof.}
If $0<\dim\HH_2<\aleph_0$, then $\s_p(T^*)\supset\s_p(B^*)\ne \emptyset$.
For every $\m\in\s_p(T^*)$, the subspace $(\ker(T^*-\m I))^\perp$ is a nontrivial hyperinvariant subspace of $T$.

If $\dim\HH_2=0$, then $\hbox{Hlat}\, T$ is isomorphic to $\hbox{Hlat}\, S$.
Furthermore, for any $n\in\N, \; (\ker S^{*n})^\perp$ is a nontrivial hyperinvariant subspace of $S$.
(Actually, every operator in $\{S\}'$ can be approximated by the polynomials of $S$ in the weak operator topology; see, e.g., Theorem 12 in \cite{KSz}.
Hence $\hbox{Hlat}S$ coincides with $\hbox{Lat}S$, which was completely characterized by Beurling.)

\rightline{$\square$}

There are no funding in connection with this paper.

\bigskip

L\'aszl\'o K\'erchy

Bolyai Institute, University of Szeged

Aradi v\'ertan\'uk tere 1

Szeged 6720

Hungary

e-mail: kerchy@math.u-szeged.hu

\end{document}